\documentclass[12pt]{amsart}

\usepackage[a4paper,margin=23mm,top=30mm]{geometry}

\usepackage{amsfonts, amsmath, amssymb, amsgen, amsthm, amscd}
\usepackage{newtxtext,newtxmath}
\usepackage[utf8]{inputenc} 

\usepackage{color}

\usepackage[all]{xy}

\usepackage{hyperref}
\usepackage{mathtools,slashed}

\usepackage{enumerate}

\def\dt{\left.\frac{d}{dt}\right|_{_{t=0}}}

\def\a{\alpha}
\def\b{\beta}
\def\d{\delta}
\def\D{\Delta}

\def\Id{\mathop{\rm Id}\nolimits}

\def\ot{\otimes}

\def\wdots{\wedge\dots\wedge}

\def\rt{\triangleright}
\def\lt{\triangleleft}

\def\D{\Delta}

\def\Id{\mathop{\rm Id}\nolimits}

\def\ad{\mathop{\rm ad}\nolimits}

\usepackage{enumerate}

\newcommand{\G}[1]{\mathfrak{#1}}
\newcommand{\C}[1]{\mathcal{#1}}
\newcommand{\B}[1]{\mathbb{#1}}

\newcommand{\Hom}{{\rm Hom}}

\renewcommand{\leq}{\leqslant}
\renewcommand{\geq}{\geqslant}

\numberwithin{equation}{section}

\newtheorem{theorem}{Theorem}[section]
\newtheorem{proposition}[theorem]{Proposition}

\newtheorem{corollary}[theorem]{Corollary}
\theoremstyle{definition}

\newtheorem{remark}[theorem]{Remark}
\newtheorem{example}[theorem]{Example}

\title{Cohomologies and generalized derivation extensions of $n$-Lie algebras}

\author{B. Ate\c{s}l\.{i}}
\address{Gebze Technical University, Kocaeli, Turkey}
\email{b.atesli@gtu.edu.tr}

\author{O. Esen}
\address{Gebze Technical University, Kocaeli, Turkey}
\email{oesen@gtu.edu.tr}

\author{S. Sütlü}
\address{I\c{s}ik University, Istanbul, Turkey}
\email{serkan.sutlu@isikun.edu.tr}

\begin{document}

\begin{abstract}
A cohomology theory, associated to a $n$-Lie algebra and a representation space of it, is introduced. It is observed that this cohomology theory is qualified to encode the generalized derivation extensions, and that it coincides, for $n=3$, with the known cohomology of $n$-Lie algebras. The abelian extensions and infinitesimal deformations of $n$-Lie algebras, on the other hand, are shown to be characterized by the usual cohomology of $n$-Lie algebras. Furthermore, the Hochschild-Serre spectral sequence of the Lie algebra cohomology is upgraded to the level of $n$-Lie algebras, and is applied to the cohomology of generalized derivation extensions. 
\end{abstract}

\maketitle

\tableofcontents

\section*{Introduction}

In an effort to generalize the classical Hamiltonian mechanics to a theory that can accommodate two Hamiltonians, ternary Lie algebras were introduced in \cite{Namb73}, wherein it was hinted that the theory may be generalized into one that allows $(n-1)$-many Hamiltonians. This point of view was pursued further in \cite{Fili85} introducing a theory of $n$-Lie algebras, and then has further been developed in \cite{Takh94}. 

It was in \cite{Gaut96}, for the first time, a cohomology theory has been associated to $n$-ary Lie algebras to study their formal deformations. Then, in \cite{daletskii1997leibniz}, this cohomology theory has been considered further, in close connection with the Leibniz cohomology. This cohomology theory, associated to $n$-Lie algebras, upgraded recently to a cohomology theory with arbitrary coefficients in \cite{bai2016bialgebras} and \cite{Zhan14} in the case of $n=3$, and in \cite{SongTang19} to a cohomology theory associated to a $n$-Hom-Lie algebra along with a representation space of it.

On the other hand, cohomology theories, more precisely the low dimensional cohomology groups, has intimate relations with the extensions/deformations of the algebraic/geometric objects, \cite{HochSerr53-I,HochSerr53}. Regarding the $n$-ary Lie algebras, this point of view has been recently taken up by \cite{Zhan14} both for the abelian extensions and the infinitesimal deformations, in the case of $n=3$. In the general case of $n$-Lie algebras, the analogue results were obtained in \cite{AzcaIzqu10}, see also \cite[Prop. 4.3]{Makh16}, for the central extensions, and in \cite[Thm. 4.2]{Makh16} for (1-parameter) formal deformations. We shall hereby consider the cases of abelian extensions of $n$-Lie algebras, by a representation space, and the infinitesimal deformations.

The case of derivation extensions, on the other hand, has been considered in \cite{SongJian18} for 3-Lie algebras, and it is argued that a derivation of a 3-Lie algebra $\G{G}$ does not fit to define a 3-Lie algebra structure on the 1-dimensional extension $\G{G}\oplus k$. The authors, therefore, introduced the notion of a generalized derivation for a 3-Lie algebra, which has later been extended to $n$-Hom-Lie algebras in \cite{SongTang19}. 

It is very well-known that the derivations on $n$-Lie algebras, Leibniz algebras, etc. are in one-to-one correspondence with the 1st cohomology group of this algebraic object in question, with coefficients in itself. However, as is pointed out in \cite{SongJian18} for 3-Lie algebras, such a characterization is missing for the case of generalized derivations. It is this gap that the present paper aims to fill in. More precisely, we realize a generalized derivations on a $n$-Lie algebra as a 1-cocycle in a cohomology theory associated to the very $n$-Lie algebra (along with a representation space) under consideration. Surprisingly enough, we observe that the cohomology theory classifying the generalized derivations on $n$-Lie algebras is not the usual one that classify the abelian extensions or the infinitesimal deformations. Nevertheless, these two cohomology theories coincides (only) in the case of $n=3$.

Finally, we upgrade the Hochschild-Serre spectral sequence for Lie algebras to $n$-Lie algebras, allowing thus a thorough homological analysis on $n$-Lie algebras. We present explicitly the $E_1$-term relative to a subalgebra, and then the $E_2$-term relative to an ideal. As a manifestation of the structure of the cohomology of $n$-Lie algebras, the spectral sequence bears the Leibniz cohomology in addition to the $n$-Lie algebra cohomology. Furthermore, using this spectral sequence we show that the cohomology of the generalized derivation extension of a $n$-Lie algebra coincides with the cohomology of the $n$-Lie algebra itself.

The paper may be outlined as follows.

The very first section consists of the preliminary material to be used in the sequel. More precisely, in Subsection \ref{subsect-Lie-alg} we recall the derivation extensions of Lie algebras with the Lie algebra cohomology. In Subsection \ref{subsect-Leibniz}, on the other hand, we recall briefly the representations and the cohomology of Leibniz algebras. In Section \ref{sect-n-Lie} we present the main results of the paper. Following the auxiliary results on the representations of $n$-Lie algebras in Subsection \ref{subsect-n-Lie-alg}, and the $n$-Lie algebra cohomology in Subsection \ref{subsect-cohom-n-Lie}, we discuss the abelian extensions of $n$-Lie algebras in Proposition \ref{prop-ab-ext} and Proposition \ref{prop-ab-ext-II} of Subsection \ref{subsect-ab-ext}, along with the infinitesimal deformations of $n$-Lie algebras in Subsection \ref{subsect-inf-deform}. Then, in Subsection \ref{subsect-gen-der-n-Lie}, we introduce the cohomology theory encoding the generalized derivation extensions of $n$-Lie algebras. More precisely, we prove Proposition \ref{prop-gen-der-1-cocycle-n-Lie}. Finally, Section \ref{sect-Hoch-Serr-spect-seq} is reserved for the Hochschild-Serre spectral sequence for the $n$-Lie algebra cohomology. As a direct application, we obtain Corollary \ref{coroll-cohom-gen-der-ext} on the cohomology of a generalized derivation extension of a $n$-Lie algebra.

\subsubsection*{Notations and conventions}~

Given a linear space $V$, we shall denote by $g\ell(V)$ the space of linear maps on $V$. On the other hand, given a $n$-Lie algebra $\C{L}$, we shall employ the notations $\C{L}_{n} := \wedge^{n}\C{L}$, $\C{L}_{n-1} := \wedge^{n-1}\C{L}$, and $\C{L}_{n-2} := \wedge^{n-2}\C{L}$.


\section{Lie and Leibniz algebras}\label{sect-Lie-Leibniz}

This introductory section is meant to review the basics of Lie and Leibniz algebras which will be needed in the sequel. To be more precise, we shall take a quick tour towards the cohomology of Lie algebras in the first subsection, whereas the second subsection will be reserved for the representations and the cohomology of Leibniz algebras. 

\subsection{Derivation extensions of Lie algebras}~\label{subsect-Lie-alg}

In the present subsection we shall present a brief overview of derivation extensions of Lie algebras, through the Lie algebra cohomology. 

Let us begin with the Lie algebra cohomology from \cite{HochSerr53}, as well as \cite{Fuks-book}. Let $\G{g}$ be a Lie algebra, and let $V$ be a (left) representation of $\G{g}$. Then, the \emph{Chevalley-Eilenberg complex}
\[
C(\G{g},V) := \bigoplus_{n\geq 0} \,C^n(\G{g},V), \qquad C^n(\G{g},V):= {\rm Hom}\left(\wedge^n\,\G{g}, V\right),
\]
together with $d:C^n(\G{g},V) \to C^{n+1}(\G{g},V)$ which is given for any $f\in C^n(\G{g},V)$ by
\begin{align*}
& df(X_1,\ldots,X_{n+1}) := \\
& \sum_{1\leq i < j \leq n+1}\,(-1)^{i+j}\,f([X_i,X_j], X_1,\ldots, \widehat{X}_i, \ldots, \widehat{X}_j, \ldots, X_{n+1}) + \\
& \sum_{1\leq k \leq n+1}\,(-1)^{k+1}\,X_k\rt f(X_1,\ldots, \widehat{X}_k,\ldots, X_{n+1}),
\end{align*}
form a differential complex. The homology $H(C(\G{g},V), d)$ is called the \emph{Lie algebra cohomology} of $\G{g}$, with coefficients in $V$, and is denoted by $H^\ast(\G{g},V)$. In other words,
\[
Z^n(\G{g},V) = \{f\in C^n(\G{g},V) \mid df = 0\}
\]
being the set of $n$-cocycles, and 
\[
B^n(\G{g},V) = \{f\in C^n(\G{g},V) \mid f = dg, \,\, {\rm for \,\,some}\,\, g\in C^{n-1}(\G{g},V)\}
\]
being the set of $n$-coboundaries,
\[
H^n(\G{g},V) := Z^n(\G{g},V) / B^n(\G{g},V), \qquad n\geq 0.
\]

Now, a \emph{derivation} of a Lie algebra $\G{g}$ is defined to be a linear operator $D:\G{g}\to \G{g}$ so that
\[
D([X,Y]) = [D(X),Y] + [X,D(Y)]
\]
for any $X,Y \in \G{g}$. It follows at once from the \emph{Jacobi identity}; 
\[
[X,[Y,Z]] + [Y,[Z,X]] + [Z,[X,Y]] = 0
\]
for all $X,Y,Z \in \G{g}$, that the \emph{adjoint action} yields a derivation $\ad_X:\G{g}\to \G{g}$, $\ad_X(Y):=[X,Y]$, for any $X,Y\in \G{g}$. Given a derivation $D:\G{g}\to \G{g}$, if $D=\ad_X$ for some $X\in \G{g}$, then it is called an \emph{inner derivation}.

Given a Lie algebra $\G{g}$, represented over itself by the adjoint action, it is well-known that $Z^1(\G{g},\G{g})$ coincides with the set $Der(\G{g})$ of derivations on $\G{g}$, while $B^1(\G{g},\G{g})$ represents the set $I(\G{g})$ of inner derivations. Accordingly,
\[
H^1(\G{g},\G{g}) = Z^1(\G{g},\G{g}) / B^1(\G{g},\G{g}) = Der(\G{g}) / I(\G{g}).
\]

Furthermore, given a linear mapping $f:\G{g}\to \G{g}$, the space $\G{g}\oplus_f k:=\G{g}\oplus k$ is a Lie algebra through 
\begin{equation}\label{der-ext}
[(X,\a),(Y,\b)] := \Big([X,Y] + \a f(Y) - \b f(X), \, 0\Big) 
\end{equation}
for any $X,Y\in \G{g}$ and any $\a,\b\in k$, if and only if $f \in Z^1(\G{g},\G{g})$.  

If, on the other hand, $f \in B^1(\G{g},\G{g})$; say $f = dZ = -\ad_Z$ for some $Z\in \G{g}$, then $\G{g}\oplus_{f} k \cong \G{g}\rtimes kZ$.

As a result, following the terminology in \cite[Sect. I.4]{Fuks-book}, the equivalence classes of \emph{1-dimensional right extensions} of a Lie algebra $\G{g}$ are in one-to-one correspondence with $H^1(\G{g},\G{g})$.

\subsection{Representations and the cohomology of Leibniz algebras}~\label{subsect-Leibniz}

In close connection with $n$-Lie algebras, for any $n\geq 2$, there are the Leibniz algebras. As such, we find it instructive to recall briefly the basics of Leibniz algebras, along with the Leibniz algebra cohomology. The concepts recalled here will be needed in the sequel. 

Let us begin with the definition of a (left) Leibniz algebra following \cite{SongJian18}. For the right handed version, we refer the reader to \cite{LodaPira93}. For a recent survey on Leibniz algebras, we refer the reader to \cite{Feld19}.

A  \emph{left Leibniz algebra} is a vector space $\G{L}$ equipped with a bilinear map $[\,,]:\G{L}\otimes \G{L} \to \G{L}$ satisfying
\begin{equation}\label{Leibniz-id-left}
[x,[y,z]] = [[x,y],z] + [y,[x,z]]
\end{equation}
for all $x,y,z \in \G{L}$, whereas a \emph{right Leibniz algebra} is one that
\begin{equation}\label{Leibniz-id-right}
[[x,y],z] = [[x,z],y] + [x,[y,z]]
\end{equation}
is satisfied.

Moreover, a Leibniz algebra is called \emph{symmetric} if it is both a left and a right Leibniz algebra. It follows at once from the definition above that a Lie algebra is a symmetric Leibniz algebra. Other examples of (left) Leibniz algebras include associative algebras equipped with suitable endomorphisms, see for instance \cite[Ex. 1.2]{LodaPira93}.

Following the ideas developed in \cite[Subsect. 1.5]{LodaPira93}, the notion of a representation of a (left) Leibniz algebra is given as in \cite[Subsect. 2.1]{SongJian18}.

A \emph{representation} of a Leibniz algebra $(\G{L}, [\,,])$ is defined to be a vector space $V$ equipped with the linear maps $\lambda:\G{L}\to g\ell(V)$ and $\rho:\G{L} \to g\ell(V) $ such that
\begin{align}
&\lambda([x,y]) =\lambda(x)\circ\lambda(y) - \lambda(y)\circ\lambda(x), \label{rep-Leibniz-I} \\
&\rho([x,y]) = \lambda(x)\circ\rho(y) - \rho(y)\circ\lambda(x), \label{rep-Leibniz-II} \\
&\rho(y) \circ \rho(x) = -\rho(y) \circ \lambda(x), \label{rep-Leibniz-III}\\
\end{align}
for any $x,y\in \G{L}$.

Let us note also that, from the point of view of (left and right) \emph{actions}, we may view $\lambda:\G{L}\to g\ell(V)$ as a left action $\rt:\G{L}\ot V\to V$ by 
\[
x\rt v := \lambda(x)(v),
\]
and $\rho:\G{L} \to g\ell(V)$ as a right action $\lt:V\ot \G{L} \to V$ via
\[
v\lt x := \rho(x)(v),
\]
for any $x,y\in \G{L}$, and any $v\in V$. Accordingly, $\lambda:\G{L} \to g\ell(V)$ may be called a \emph{left representation} of the Leibniz algebra, while $\rho:\G{L} \to g\ell(V)$ is referred as a \emph{right representation}. Hence, the conditions \eqref{rep-Leibniz-I}-\eqref{rep-Leibniz-III} may be re-written as
\begin{align}
& [x,y] \rt v = x\rt (y\rt v) - y\rt (x\rt v), \label{rep-1}\\
& v\lt [x,y] = x\rt (v\lt y) - (x\rt v) \lt y, \label{rep-2} \\
& (v\lt x)\lt y = - (x\rt v)\lt y. \label{rep-3}
\end{align}

Given a Leibniz algebra $\G{L}$, 
\begin{equation}\label{ad-left}
\ad^L:\G{L}\to g\ell(\G{L}), \qquad \ad^L(x)(y) := [x,y]
\end{equation}
is a left representation of the Leibniz algebra $\G{L}$ onto itself, called the \emph{left adjoint representation} of $\G{L}$ on itself, while
\begin{equation}\label{ad-right}
\ad^R:\G{L}\to g\ell(\G{L}), \qquad \ad^R(x)(y) := [y,x]
\end{equation}
is a right representation, called the \emph{right adjoint representation} of $\G{L}$.

We shall make use of the notations
\[
\ad^L_x := \ad^L(x), \qquad \ad^R_x := \ad^R(x)
\]
for any $x\in \G{L}$. 

A representation $V$ of a Leibniz algebra $\G{L}$ is called \emph{symmetric} if
\[
x \rt v + v\lt x = 0
\] 
for any $x\in \G{L}$, and any $v\in V$, and $V$ is called \emph{anti-symmetric} if 
\[
v \lt x = 0.
\]

Given a Leibniz algebra $\G{L}$, and a representation $V$ of $\G{L}$, let 
\[
V_{anti} := {\rm Span}\{x\rt v + v\lt x \mid x\in \G{L},\,\,v\in V\}
\]
be the \emph{anti-symmetric kernel} of $V$, \cite{Feld19,LodaPira93}. Then, the \emph{symmetrization} 
\[
V_{sym} := V / V_{anti}
\]
of $V$, is a symmetric representation of $\G{L}$.

On the other extreme, we may consider the symmetric sub-representations.

\begin{example}\label{ex-rep-sym}
Given a representation $V$ of a Leibniz algebra $\G{L}$,
\[
V^{sym} := \{v \in V\mid x\rt v + v\lt x = 0, \,\, \forall\,x\in \G{L}\}
\]
is also a representation. Indeed, given any $w\in V^{sym}$, and any $x,y\in \G{L}$,
\begin{align*}
& y\rt (x\rt w) + (x\rt w)\lt y = y\rt (x\rt w) + x\rt (w \lt y) - w \lt [x,y] = \\
& y\rt (x\rt w) - x\rt (y \rt w) - w \lt [x,y] = - [x,y]\rt w - w \lt [x,y] = 0,
\end{align*}
that is, $x\rt w \in V^{sym}$. Then,
\begin{align*}
& y\rt (w\lt x) + (w\lt x)\lt y = (y \rt w)\lt x + w\lt [y,x] + (w\lt x) \lt y = \\
& (y \rt w)\lt x - [y,x] \rt w + (w\lt x) \lt y = (y \rt w)\lt x - y\rt (x\rt w) + x\rt (y\rt w) + (w\lt x) \lt y = \\
& - y\rt (x\rt w)  + (w\lt x) \lt y = - y\rt (x\rt w)  - (x\rt w) \lt y = 0,
\end{align*}
where we used \eqref{rep-2} in the first equality, and on the third equality we used the result that $y\rt w \in V^{sym}$. The fourth equality, on the other hand, follows from \eqref{rep-3}, and the last equality is a result of $x\rt w\in V^{sym}$. Accordingly, we conclude that $w\lt x \in V^{sym}$.

Needless to say that $V^{sym}$ is a symmetric representation of $\G{L}$. Furthermore, it follows from its very definition that $V^{sym} = V$ if $V$ is a symmetric representation, and that $V^{sym}$ is trivial if $V$ is an anti-symmetric representation.
\end{example}

Let us next recall the semi-direct sum construction on the Leibniz algebras. Given a Leibniz algebra $\G{L}$, and a representation $V$ of $\G{L}$, it takes a routine verification that $V\rtimes \G{L} := V\oplus \G{L}$ is also Leibniz algebra via
\[
\Big[(v,x),(w,y)\Big] := \Big(x\rt w + v\lt y,\,[x,y]\Big),
\]
called the \emph{semi-direct sum} Leibniz algebra.

Given a (left) Leibniz algebra $\G{L}$, a linear map $D:\G{L} \to \G{L}$ is called a \emph{derivation} if 
\[
D([x,y]) = [D(x),y] + [x,D(y)]
\]
for all $x,y \in \G{L}$.

Given any left Leibniz algebra $\G{L}$, and any $x\in \G{L}$, the (left) adjoint action $\ad_x^L:\G{L} \to \G{L}$ is a derivation.

It then follows at once from the definition that if $\G{L}$ is a left Leibniz algebra, then \eqref{ad-left} is a derivation. Similarly, in the case of $\G{L}$ being a right Leibniz algebra,  \eqref{ad-right} is a derivation.

Let us finally recall the Leibniz algebra cohomology from \cite{FeldWage21,LodaPira93}. Let $\G{L}$ be a Leibniz algebra, and let $V$ be a representation of $\G{L}$. Then, 
\[
CL(\G{L},V) := \bigoplus_{n\geq 0} \,CL^n(\G{L},V), \qquad CL^n(\G{L},V):= {\rm Hom}\left(\G{L}^{\ot\,n}, V\right),
\]
together with $d:CL^n(\G{L},V) \to CL^{n+1}(\G{L},V)$ which is given by
\begin{align}\label{Leibniz-diff}
\begin{split}
& df(x_1,\ldots,x_{n+1}) := \\
& \sum_{1\leq i < j \leq n+1}\,(-1)^{i}\,f(x_1,\ldots, \widehat{x}_i, \ldots, x_{j-1}, [x_i,x_j], \ldots, x_{n+1}) + \\
& \sum_{k=1}^n\,(-1)^{k+1}\,x_k\rt f(x_1,\ldots, \widehat{x}_k,\ldots, x_{n+1}) + (-1)^{n+1}\,f(x_1,\ldots,x_n) \lt x_{n+1}
\end{split}
\end{align}
for any $f\in CL^n(\G{L},V)$, form a differential complex. The homology $H(CL(\G{L},V), d)$ is called the \emph{Leibniz algebra cohomology} of $\G{L}$, with coefficients in $V$, and is denoted by $HL^\ast(\G{L},V)$. Similar to the Lie algebra cohomology complex,
\[
ZL^n(\G{L},V) := \{f\in CL^n(\G{L},V) \mid df = 0\}
\]
denotes the set of $n$-cocycles, and 
\[
BL^n(\G{L},V) := \{f\in CL^n(\G{L},V) \mid f = dg, \,\, {\rm for \,\,some}\,\, g\in CL^{n-1}(\G{L},V)\}
\]
is the set of $n$-coboundaries. Accordingly,
\[
HL^n(\G{L},V) := ZL^n(\G{L},V) / BL^n(\G{L},V), \qquad n\geq 0.
\]

\section{Generalized derivation extensions of $n$-Lie algebras}\label{sect-n-Lie}

In the present section we shall present the cohomological classification of generalized derivations of a $n$-Lie algebras. In the case of $n=2$, the result recovers the analogue result recalled briefly in Subsection \ref{subsect-Lie-alg}, while in the case of $n=3$ it answers the question raised in \cite[Rk. 3.7]{SongJian18}. More precisely, we shall realize a generalized derivation of a $n$-Lie algebra as a 1-cocycle in a cohomology theory associated to $n$-Lie algebras, and an inner generalized derivation as a 1-coboundary. 

\subsection{$n$-Lie algebras}\label{subsect-n-Lie-alg}~

Let us first recall the very basics of $n$-Lie algebras from \cite{daletskii1997leibniz,Namb73,Takh94}.

A linear space $\C{L}$, equipped with a linear operator $[\,,\ldots\,,\,]:\C{L}_n\to \C{L}$, is called an $n$-Lie algebra if it satisfies the \emph{fundamental identity}
\begin{equation}\label{fund-ID-n-Lie}
[x_1,\ldots, x_{n-1}, [x_n,\ldots,x_{2n-1}]] = \sum_{k=0}^{n-1}\, [x_n, \ldots,  [x_1,\ldots,x_{n-1}, x_{n+k}],\ldots, x_{2n-1}],
\end{equation}
for any $x_1,\ldots,x_{2n-1} \in \C{L}$. 

Along the lines of \cite{daletskii1997leibniz}, given an $n$-Lie algebra $\C{L}$, the space $\C{L}_{n-1}$ has the structure of a Leibniz algebra through
\[
[x_1, x_2] := \sum_{k=1}^{n-1}\,x_2^1 \wdots [x_1^1,\ldots, x_1^{n-1},x_2^k] \wdots x_2^{n-1}
\]
for any $x_1:= x_1^1\wdots x_1^{n-1}, x_2:=x_2^1\wdots x_2^{n-1} \in \C{L}_{n-1}$. Furthermore, as was shown in \cite[Thm. 2]{daletskii1997leibniz}, $\C{L}$ happens to be a representation space for the Leibniz algebra $\C{L}_{n-1}$.

\subsubsection{Representations}~

Let us recall the representations of $n$-Lie algebras from \cite{Dzhu04,Dzhu04-II}.

Let $\C{L}$ be a $n$-Lie algebra. A vector space $V$ with a linear map $\mu:\C{L}_{n-1}\to g\ell(V)$ is called a \emph{representation} of $\C{L}$, if
\begin{align}
& \mu(x_1,\ldots,x_{n-1})\mu(x_n,\ldots,x_{2n-2}) - \mu(x_n,\ldots,x_{2n-2})\mu(x_1,\ldots,x_{n-1}) = \label{rep-n-Lie-I}\\
& \hspace{3cm} \sum_{k=0}^{n-2}\,\mu(x_n,\ldots , [x_1,\ldots,x_{n-1},x_{n+k}], \ldots, x_{2n-2}), \notag\\
& \mu([x_1,\ldots,x_n],x_{n+1},\ldots,x_{2n-2}) = \sum_{k=1}^n\,(-1)^k\,\mu(x_1,\ldots, \hat{x}_k,\ldots,x_n)\mu(x_{n+1},\ldots,x_{2n-2},x_k), \label{rep-n-Lie-II}
\end{align}
for any $x_1,\ldots,x_{2n-2}\in \C{L}$.

Let us first note that in the particular case of $n=3$, the conditions \eqref{rep-n-Lie-I} and \eqref{rep-n-Lie-II} reduce precisely to those in, for instance, \cite[Def. 2.3]{bai2016bialgebras}.

As was mentioned in \cite{Dzhu04,Dzhu04-II}, the conditions \eqref{rep-n-Lie-I} and \eqref{rep-n-Lie-II} are equivalent to $V\rtimes \C{L}:=V\oplus \C{L}$ to be a $n$-Lie algebra through
\[
[v_1+x_1, \ldots, v_n+x_n] := \sum_{k=1}^n\,\mu(x_1,\ldots, \hat{x}_k,\ldots,x_n)(v_k)+ [x_1,\ldots, x_n].
\]
In this case, the $n$-Lie algebra $V\rtimes \C{L}$ is called the ``semi-direct sum'' of $\C{L}$ and $V$.

The most immediate example of a representation is $\ad:\C{L}_{n-1}\to g\ell(\C{L})$, that is, the \emph{adjoint representation} of a $n$-Lie algebra $\C{L}$ on itself, which is given by
\begin{equation}\label{ad-n-Lie}
\ad(x_1,\ldots,x_{n-1})(x_n) := [x_1,\ldots,x_n]
\end{equation}
for any $x_1,\ldots,x_n \in \C{L}$.

A slightly more serious example, in the spirit of a dual representation, is the following. To this end, we shall first record the following generalization of \cite[Prop. 2.5(b)]{bai2016bialgebras} for the representations of $n$-Lie algebras.

\begin{proposition}
Given a $n$-Lie algebra $\C{L}$, and a representation $(V,\mu)$ of $\C{L}$, 
\begin{align}\label{property-rep-n-Lie}
\begin{split}
& \sum_{k=1}^{n}\,(-1)^k\,\mu(x_1,\ldots, \widehat{x_k},\ldots,x_n)\mu(x_{n+1},\ldots,x_{2n-2},x_k) + \\
& \hspace{3cm}\sum_{k=1}^{n}\,(-1)^k\,\mu(x_{n+1},\ldots,x_{2n-2},x_k)\mu(x_1,\ldots, \widehat{x_k},\ldots,x_n) = 0
\end{split}
\end{align}
for any $x_1,\ldots, x_{2n-2} \in \C{L}$.
\end{proposition}

\begin{proof}
To begin with, in view of \eqref{rep-n-Lie-I} we have
\begin{align*}
&\sum_{k=1}^{n}\,(-1)^k\,\mu(x_1,\ldots, \widehat{x_k},\ldots,x_n)\mu(x_{n+1},\ldots,x_{2n-2},x_k) + \\
& \hspace{3cm}\sum_{k=1}^{n}\,(-1)^k\,\mu(x_{n+1},\ldots,x_{2n-2},x_k)\mu(x_1,\ldots, \widehat{x_k},\ldots,x_n) = \\
& \sum_{k=1}^{n}\,(-1)^k\,\mu(x_{n+1},\ldots,x_{2n-2},x_k)\mu(x_1,\ldots, \widehat{x_k},\ldots,x_n) + \\
& \sum_{p=1}^{n-2}\sum_{k=1}^{n}\,(-1)^k\,\mu(x_{n+1},\ldots, [x_1,\ldots, \widehat{x_k},\ldots,x_n, x_{n+p}], \ldots, x_{2n-2},x_k) + \\
& \sum_{k=1}^{n}\,(-1)^k\,\mu(x_{n+1},\ldots,x_{2n-2},[x_1,\ldots, \widehat{x_k},\ldots,x_n,x_k]) + \\
& \sum_{k=1}^{n}\,(-1)^k\,\mu(x_1,\ldots, \widehat{x_k},\ldots,x_n)\mu(x_{n+1},\ldots,x_{2n-2},x_k) + \\
& \sum_{1\leq r < k \leq n}\sum_{k=1}^{n}\,(-1)^k\,\mu(x_1,\ldots, [x_{n+1},\ldots,x_{2n-2},x_k,x_r],\ldots, \widehat{x_k},\ldots,x_n) + \\
& \sum_{1\leq k < r\leq n}\sum_{k=1}^{n}\,(-1)^k\,\mu(x_1,\ldots, \widehat{x_k},\ldots, [x_{n+1},\ldots, x_{2n-2},x_k,x_r],\ldots,x_n).
\end{align*} 
That is,
\begin{align*}
& n\,\mu([x_1,\ldots, x_n],x_{n+1},\ldots,x_{2n-2}) + \\
& \sum_{p=1}^{n-2}\sum_{k=1}^{n}\,(-1)^{k+p-1}\,\mu([x_1,\ldots, \widehat{x_k},\ldots,x_n, x_{n+p}], x_{n+1}, \ldots, \widehat{x_{n+p}},\ldots, x_{2n-2},x_k) + \\
& \sum_{1\leq r < k \leq n}\sum_{k=1}^{n}\,(-1)^{k+r-1}\,\mu([x_{n+1},\ldots,x_{2n-2},x_k,x_r], x_1,\ldots, \widehat{x_r},\ldots, \widehat{x_k},\ldots,x_n) + \\
& \sum_{1\leq k < r\leq n}\sum_{k=1}^{n}\,(-1)^{k+r}\,\mu([x_{n+1},\ldots, x_{2n-2},x_k,x_r],x_1,\ldots, \widehat{x_k},\ldots, \widehat{x_r},\ldots,x_n) = 0.
\end{align*}
We next apply \eqref{rep-n-Lie-II} to arrive
\begin{align*}
& n\,\mu([x_1,\ldots, x_n],x_{n+1},\ldots,x_{2n-2}) + \\
& \sum_{1\leq r < k \leq n}\sum_{p=1}^{n-2}\sum_{k=1}^{n}\,(-1)^{r+k+p-1}\,\mu(x_1,\ldots, \widehat{x_r},\ldots, \widehat{x_k},\ldots,x_n, x_{n+p})\mu(x_{n+1}, \ldots, \widehat{x_{n+p}},\ldots, x_{2n-2},x_k,x_r) + \\
& \sum_{1\leq k < r \leq n}\sum_{p=1}^{n-2}\sum_{k=1}^{n}\,(-1)^{r+k+p}\,\mu(x_1,\ldots, \widehat{x_k},\ldots, \widehat{x_r},\ldots,x_n, x_{n+p})\mu(x_{n+1}, \ldots, \widehat{x_{n+p}},\ldots, x_{2n-2},x_k,x_r) + \\
& \sum_{p=1}^{n-2}\sum_{k=1}^{n}\,(-1)^{n+k+p-1}\,\mu(x_1,\ldots, \widehat{x_k},\ldots,x_n)\mu(x_{n+1}, \ldots, \widehat{x_{n+p}},\ldots, x_{2n-2},x_k,x_{n+p}) +\\
& \sum_{1\leq r < k \leq n}\sum_{p=1}^{n-2}\sum_{k=1}^{n}\,(-1)^{p+k+r-1}\,\mu(x_{n+1}, \ldots \widehat{x_{n+p}},\ldots,x_{2n-2},x_k,x_r)\mu(x_1,\ldots, \widehat{x_r},\ldots, \widehat{x_k},\ldots,x_n,x_{n+p}) + \\
& \sum_{1\leq r < k \leq n}\sum_{k=1}^{n}\,(-1)^{n+k+r}\,\mu(x_{n+1},\ldots,x_{2n-2},x_r)\mu(x_1,\ldots, \widehat{x_r},\ldots, \widehat{x_k},\ldots,x_n,x_k) + \\
& \sum_{1\leq r < k \leq n}\sum_{k=1}^{n}\,(-1)^{n+k+r-1}\,\mu(x_{n+1},\ldots,x_{2n-2},x_k)\mu(x_1,\ldots, \widehat{x_r},\ldots, \widehat{x_k},\ldots,x_n,x_r) + \\
& \sum_{1\leq k < r\leq n}\sum_{p=1}^{n-2}\sum_{k=1}^{n}\,(-1)^{k+r+p}\,\mu(x_{n+1},\ldots, \widehat{x_{n+p}},\ldots, x_{2n-2},x_k,x_r)\mu(x_1,\ldots, \widehat{x_k},\ldots, \widehat{x_r},\ldots,x_n,x_{n+p}) + \\
& \sum_{1\leq k < r\leq n}\sum_{k=1}^{n}\,(-1)^{n+k+r-1}\,\mu(x_{n+1},\ldots, x_{2n-2},x_r)\mu(x_1,\ldots, \widehat{x_k},\ldots, \widehat{x_r},\ldots,x_n,x_k) + \\
& \sum_{1\leq k < r\leq n}\sum_{k=1}^{n}\,(-1)^{n+k+r}\,\mu(x_{n+1},\ldots, x_{2n-2},x_k)\mu(x_1,\ldots, \widehat{x_k},\ldots, \widehat{x_r},\ldots,x_n,x_r)=0,
\end{align*}
which, calling the left hand side of \eqref{property-rep-n-Lie} as $F(x_1,\ldots, x_{2n-2})$, may be arranged into
\begin{align*}
& (2n-2)\,F(x_1,\ldots, x_{2n-2}) + \\
& \sum_{1\leq r < k \leq n}\sum_{p=1}^{n-2}\sum_{k=1}^{n}\,(-1)^{r+k+p-1}\,\mu(x_1,\ldots, \widehat{x_r},\ldots, \widehat{x_k},\ldots,x_n, x_{n+p})\mu(x_{n+1}, \ldots, \widehat{x_{n+p}},\ldots, x_{2n-2},x_k,x_r) + \\
& \sum_{1\leq k < r \leq n}\sum_{p=1}^{n-2}\sum_{k=1}^{n}\,(-1)^{r+k+p}\,\mu(x_1,\ldots, \widehat{x_k},\ldots, \widehat{x_r},\ldots,x_n, x_{n+p})\mu(x_{n+1}, \ldots, \widehat{x_{n+p}},\ldots, x_{2n-2},x_k,x_r) + \\
& \sum_{1\leq r < k \leq n}\sum_{p=1}^{n-2}\sum_{k=1}^{n}\,(-1)^{p+k+r-1}\,\mu(x_{n+1}, \ldots \widehat{x_{n+p}},\ldots,x_{2n-2},x_k,x_r)\mu(x_1,\ldots, \widehat{x_r},\ldots, \widehat{x_k},\ldots,x_n,x_{n+p}) + \\
& \sum_{1\leq k < r\leq n}\sum_{p=1}^{n-2}\sum_{k=1}^{n}\,(-1)^{k+r+p}\,\mu(x_{n+1},\ldots, \widehat{x_{n+p}},\ldots, x_{2n-2},x_k,x_r)(x_1,\ldots, \widehat{x_k},\ldots, \widehat{x_r},\ldots,x_n,x_{n+p})=0.
\end{align*}
Next, setting for any $a\in\{1,\ldots,n\}$ and $b\in\{1,\ldots,n-2\}$
\[
F_{a,b}:=F(x_1,\ldots,x_{a-1},x_{n+b},x_{a+1},\ldots, x_n,x_{n+1},\ldots, x_{n+b-1},x_a,x_{n+b+1},\ldots,x_{2n-2}),
\]
the above expression may be presented as
\begin{equation}\label{counting-formula}
\sum_{k=1}^n\sum_{p=1}^{n-2}\,F_{k,p} = (n^2-2)F(x_1,\ldots, x_{2n-2}).
\end{equation}
The claim, now, follows from the iteration of \eqref{counting-formula}, or equivalently the successive applications of \eqref{rep-n-Lie-I} and \eqref{rep-n-Lie-II}.
\end{proof}

\begin{remark}
Let us note that in the case of $n=3$, \eqref{property-rep-n-Lie} is precisely \cite[Prop. 2.5(b)]{bai2016bialgebras} for 3-Lie algebras.
\end{remark}

\begin{example}\label{rep-Hom(V,W)}
Let $\C{L}$ be a $n$-Lie algebra, together with a representation $(V,\eta)$. Then, the space $\Hom(V,W)$ of linear maps from $V$ to a vector space $W$ is also a representation of $\C{L}$ via
\begin{equation}\label{ex-adj-rep-linear-map-n-Lie}
\mu: \C{L}_{n-1} \to g\ell(\Hom(V,W)), \qquad \Big(\mu(x_1,\ldots,x_{n-1})(T)\Big)(v) := -T(\eta(x_1,\ldots,x_{n-1})(v)).
\end{equation}
Indeed, we observe at once that
\begin{align}\label{Lie-derivative}
\begin{split}
& \Big(\mu(x_1,\ldots,x_{n-1})\mu(x_n,\ldots,x_{2n-2})(T) - \mu(x_n,\ldots,x_{2n-2})\mu(x_1,\ldots,x_{n-1})(T)\Big)(v) = \\
& T(\eta(x_n,\ldots,x_{2n-2})\eta(x_1,\ldots,x_{n-1})(v) - \eta(x_1,\ldots,x_{n-1})\eta(x_n,\ldots,x_{2n-2})(v)) = \\
& -\sum_{k=0}^{n-2}\,T(\eta(x_n,\ldots , [x_1,\ldots,x_{n-1},x_{n+k}], \ldots, x_{2n-2})(v)) = \\
& \sum_{k=0}^{n-2}\,\Big(\mu(x_n,\ldots , [x_1,\ldots,x_{n-1},x_{n+k}], \ldots, x_{2n-2})(T)\Big)(v),
\end{split}
\end{align}
where we employed \eqref{rep-n-Lie-I} on the second equality. 

On the other hand,
\begin{align*}
& \Big(\mu([x_1,\ldots,x_n],x_{n+1},\ldots,x_{2n-2})(T)\Big)(v) =  - T(\eta([x_1,\ldots,x_n],x_{n+1},\ldots,x_{2n-2})(v)) = \\
& -\sum_{k=1}^n\,(-1)^k\,T(\eta(x_1,\ldots, \hat{x}_k,\ldots,x_n)\eta(x_{n+1},\ldots,x_{2n-2},x_k)(v)) = \\
& \sum_{k=1}^n\,(-1)^k\,T(\eta(x_{n+1},\ldots,x_{2n-2},x_k)\eta(x_1,\ldots, \hat{x}_k,\ldots,x_n)(v)) = \\
& \sum_{k=1}^n\,(-1)^k\,\Big(\mu(x_1,\ldots, \hat{x}_k,\ldots,x_n)\mu(x_{n+1},\ldots,x_{2n-2},x_k)(T)\Big)(v),
\end{align*}
where we applied \eqref{property-rep-n-Lie} on the third equality. 
\end{example}

For a direct proof, one that does not appeal to an analogue of \cite[Prop. 2.5(b)]{bai2016bialgebras}, we refer the reader to \cite[Thm. 3.9]{SongTang19}.

\subsubsection{Generalized derivations}~

In accordance with the 3-Lie algebras, given a $n$-Lie algebra $\C{L}$, a linear operator $D:\C{L}\to \C{L}$ is called a \emph{derivation} if 
\[
D([x_1,\ldots,x_n]) = \sum_{k=1}^n\,[x_1,\ldots, D(x_k),\ldots, x_n]
\] 
for any $x_1,\ldots, x_n \in \C{L}$. Accordingly, it follows at once from the fundamental identity \eqref{fund-ID-n-Lie} that 
\[
\ad(x_1,\ldots,x_{n-1}):\C{L}\to \C{L}
\]
given by \eqref{ad-n-Lie} is a derivation for any $x_1,\ldots, x_{n-1} \in \C{L}$. Such derivations are called \emph{inner derivations}.

A \emph{generalized derivation} of a $n$-Lie algebra $\C{L}$, on the other hand, is defined to be a linear map $D:\C{L}_{n-1}\to \C{L}$ satisfying
\begin{align}
& D(x_1,\ldots, x_{n-2}, [x_{n-1},\ldots,x_{2n-2}]) = \sum_{k=0}^{n-1}\,[x_{n-1},\ldots,D(x_1,\ldots, x_{n-2},x_{n-1+k}), \ldots, x_{2n-2}] , \label{gen-der-n-Lie-I}\\
& [x_1,\ldots, x_{n-1},D(x_n,\ldots,x_{2n-2})] + (-1)^n [D(x_1,\ldots,x_{n-1}), x_n,\ldots,x_{2n-2}] =  \label{gen-der-n-Lie-II}\\
& \hspace{2cm}\sum_{k=0}^{n-2}\, D(x_n,\ldots,[x_1,\ldots,x_{n-1},x_{n+k}], \ldots, x_{2n-2}), \notag\\
& D(x_1,\ldots,x_{n-2},D(x_{n-1},\ldots,x_{2n-3})) = \sum_{k=0}^{n-2}\,D(x_{n-1},\ldots, D(x_1,\ldots,x_{n-2},x_{n-1+k}),\ldots, x_{2n-2}), \label{gen-der-n-Lie-III}
\end{align}
for any $x_1,\ldots,x_{n-1}\in \C{L}$, see \cite[Def. 5.1]{SongTang19}.

It follows at once, as is noted in \cite[Lemma 5.3]{SongTang19}, that for any $y\in \C{L}$ $\ad_y:\C{L}_{n-1}\to\C{L}$, $\ad_y(z_1,\ldots, z_{n-1}):= [y,z_1,\ldots, z_{n-1}]$, is a generalized derivation. Such a generalized derivation is called an \emph{inner generalized derivation}.

Let us note also that $D:\C{L}_{n-1}\to \C{L}$ being a generalized derivation, $\C{L}\oplus_D k:= \C{L}\oplus k$ is a $n$-Lie algebra through
\begin{equation}
[x_1+\a_1, \ldots, x_n+\a_n] :=  [x_1,\ldots,x_n] + \sum_{k=1}^n\,(-1)^{k+1}\a_kD(x_1,\ldots,\widehat{x_k},\ldots,x_n).
\end{equation}

Remarks on the conditions \eqref{gen-der-n-Lie-I} and \eqref{gen-der-n-Lie-III} are in order now.

\begin{remark}
Setting 
\begin{equation}\label{D-sharp-derivation}
D^\sharp:\C{L}_{n-2} \to g\ell(\C{L}), \qquad D^\sharp(x_1,\ldots,x_{n-2})(y) := D(x_1,\ldots,x_{n-2},y),
\end{equation}
we see at once that \eqref{gen-der-n-Lie-I} is equivalent to $D^\sharp(x_1,\ldots,x_{n-2})\in g\ell(\C{L})$ being a derivation of the $n$-Lie algebra $\C{L}$, for any $x_1,\ldots,x_{n-2},y\in\C{L}$.
\end{remark}

\begin{remark}
The condition \eqref{gen-der-n-Lie-III} is equivalent for the $n$-Lie algebra $\C{L}$ to have the structure of a $(n-1)$-Lie algebra through
\[
D:\C{L}_{n-1}\to \C{L}, \qquad x_1\wdots x_{n-1}\mapsto D(x_1,\ldots, x_{n-1}).
\]
\end{remark}

As for the condition \eqref{gen-der-n-Lie-II}, we have the following analogue of \cite[Lemma 3.5]{SongJian18}. Let us, however, record first the semi-direct sum Leibniz algebra $\C{L}\rtimes \C{L}_{n-1}:= \C{L}\oplus \C{L}_{n-1}$ associated to a $n$-Lie algebra $\C{L}$.

To begin with, $\C{L}$ being an adjoint representation over itself via \eqref{ad-n-Lie}, it becomes a representation of the Leibniz algebra $\C{L}_{n-1}$ through
\[
x \rt y := [x^1,\ldots, x^{n-1},y], \qquad y \lt x := - [x^1,\ldots, x^{n-1},y],
\]
for any $y\in \C{L}$ and any $x:=x^1\wdots x^{n-1}\in \C{L}_{n-1}$. Accordingly, $\C{L}\oplus \C{L}_{n-1}$ is a Leibniz algebra, the bracket on which being
\[
[y_1+x_1, y_2 + x_2] := [x_1^1,\ldots,x_1^{n-1},y_2] - [x_2^1,\ldots,x_2^{n-1},y_1] + [x_1,x_2],
\]
for any $y_1, y_2 \in \C{L}$, and any $x_1:= x_1^1\wdots x_1^{n-1}, x_2:= x_2^1\wdots x_2^{n-1} \in \C{L}_{n-1}$. Let us note that
\[
[x_1,x_2] := \sum_{k=1}^{n-1}\,x_2^1\wdots [x_1^1,\ldots,x_1^{n-1},x_2^k]\wdots x_2^{n-1} \in \C{L}_{n-1}.
\]

\begin{proposition}
Given a $n$-Lie algebra $\C{L}$, let $D:\C{L}_{n-1}\to \C{L}$ be a generalized derivation. Then,
\[
D:\C{L}\rtimes \C{L}_{n-1} \to \C{L}\rtimes \C{L}_{n-1}, \qquad D(y + x) := D(x),
\]
for any $y\in \C{L}$ and any $x:=x^1\wdots x^{n-1}\in \C{L}_{n-1}$, is a derivation of Leibniz algebras.
\end{proposition}

\begin{proof}
On the one hand we have
\begin{align*}
& D([y_1+x_1,y_2+x_2]) = D([x_1^1,\ldots,x_1^{n-1},y_2] - [x_2^1,\ldots,x_2^{n-1},y_1] + [x_1,x_2]) = \\
& \sum_{k=1}^{n-1}\,D(x_2^1,\ldots, [x_1^1,\ldots,x_1^{n-1},x_2^k],\ldots, x_2^{n-1}), 
\end{align*}
while on the other hand,
\begin{align*}
& [D(y_1+x_1),y_2+x_2] + [y_1+x_1,D(y_2+x_2)] = \\
& [D(x_1),y_2+x_2] + [y_1+x_1,D(x_2)] = \\
& [x_1^1,\ldots,x_1^{n-1}, D(x_2^1,\ldots,x_2^{n-1})] - [x_2^1,\ldots,x_2^{n-1}, D(x_1^1,\ldots,x_1^{n-1})] = \\
& [x_1^1,\ldots,x_1^{n-1}, D(x_2^1,\ldots,x_2^{n-1})] + (-1)^n [D(x_1^1,\ldots,x_1^{n-1}),x_2^1,\ldots,x_2^{n-1}].
\end{align*}
The equality, and hence the claim, now follows from \eqref{gen-der-n-Lie-II}.
\end{proof}

\subsection{Cohomology of $n$-Lie algebras}\label{subsect-cohom-n-Lie}~

In order to classify the derivation extensions of $n$-Lie algebras, via cohomology, we shall now recall the cohomology theory for $n$-Lie algebras, developed in \cite{daletskii1997leibniz}, to a cohomology theory with arbitrary coefficients, see for instance \cite[Sect. 3]{SongTang19}. 

Let $\C{L}$ be a $n$-Lie algebra, and let $(V,\mu)$ be a representation of $\C{L}$. Setting $x_i := x_i^1\wdots x_i^{n-1} \in \C{L}_{n-1}$,  
\begin{equation}\label{n-Lie-diff-comp}
C(\C{L},V) := \bigoplus_{m\geq 0} \,C^m(\C{L},V), \qquad C^m(\C{L},V):= \Hom\Big(\C{L}_{n-1}^{\wedge\,(m-1)}\wedge \C{L}, V\Big), \quad m\geq 1,
\end{equation}
where $C^0(\C{L},V):= \C{L}_{n-2}\ot V$, is a differential complex via
\begin{align}\label{n-Lie-diff}
\begin{split}
& \delta:C^0(\C{L},V) \to C^1(\C{L},V), \qquad \d(z_1\wdots z_{n-2}\ot v) (y) := \mu(z_1,\ldots, z_{n-2},y)(v), \\
& \delta:C^m(\C{L},V) \to C^{m+1}(\C{L},V), \qquad m\geq 1\\
& \delta f(x_1,\ldots, x_m,y)  := \\
& \sum_{1\leq i < j \leq m}\, (-1)^i\, f(x_1,\ldots,\widehat{x_i},\ldots,  [x_i,x_j], \ldots, x_m,y) +\\
 & \hspace{3cm}  \sum_{i=1}^{m}\, (-1)^i\, f(x_1,\ldots,\widehat{x_i},\ldots,x_m,[x_i^1,\ldots,x_i^{n-1},y]) +\\
 &\hspace{3cm} \sum_{i=1}^{m}\, (-1)^{i+1} \,\mu(x_i^1,\ldots,x_i^{n-1})\Big( f(x_1,\ldots,\widehat{x_i},\ldots,x_m,y)\Big) +\\
 & \hspace{3cm}  \sum_{i=1}^{n-1}\,(-1)^{n-1+m+i}\, \mu(x_m^1,\ldots,\widehat{x_m^i},\ldots,x_m^{n-1},y) \Big(f(x_1,\ldots,x_{m-1}, x_m^i)\Big),
 \end{split}
\end{align}
for any $f\in C^m(\C{L},V)$.

We shall now present, following \cite{daletskii1997leibniz}, the relation between the cohomology of a $n$-Lie algebra $\C{L}$, and the cohomology of the associated Leibniz algebra $\C{L}_{n-1}$. 

To this end, we shall need the following auxiliary result on the representation of the Leibniz algebra associated to the given $n$-Lie algebra.

\begin{proposition}\label{prop-L-n-2-ot-V-rep-L-n-1}
Given a $n$-Lie algebra $\C{L}$, together with a representation $(V,\mu)$, the linear space $\C{L}_{n-2}\ot V$ is a representation of the Leibniz algebra $\C{L}_{n-1}$ through
\begin{align}
& x\rt (z_1\wdots z_{n-2}\ot v) := \label{rep-Leibniz-n-Lie-I}\\
& \hspace{1cm} \sum_{k=1}^{n-2}\,z_1\wdots [x^1,\ldots, x^{n-1}, z_k]\wdots z_{n-2}\ot v + z_1\wdots z_{n-2}\ot \mu(x^1,\ldots, x^{n-1})(v), \notag\\
& (z_1\wdots z_{n-2}\ot v) \lt x := \sum_{k=1}^{n-1}\,(-1)^{n+k}\, x^1\wdots \widehat{x^k}\wdots x^{n-1}\ot \mu(z_1,\ldots, z_{n-2}, x^k)(v). \label{rep-Leibniz-n-Lie-II}
\end{align}
\end{proposition}

\begin{proof}
Let us begin with the verification of \eqref{rep-1}. To this end, we note from \eqref{rep-Leibniz-n-Lie-I} that
\begin{align*}
&y\rt (x\rt (z_1\wdots z_{n-2}\ot v) )= \\
& \sum_{k=1}^{n-2}\,y\rt \Big(z_1\wdots [x^1,\ldots, x^{n-1}, z_k]\wdots z_{n-2}\ot v\Big) +\\
&\hspace{3cm}  y\rt \Big(z_1\wdots z_{n-2}\ot \mu(x^1,\ldots, x^{n-1})(v)\Big) = \\
& \sum_{k=1}^{n-2}\underset{p\neq k}{\sum_{p=1}^{n-2}}\, z_1\wdots [y^1,\ldots, y^{n-1}, z_p] \wdots [x^1,\ldots, x^{n-1}, z_k]\wdots z_{n-2}\ot v + \\
& \sum_{k=1}^{n-2}\, z_1\wdots [y^1,\ldots, y^{n-1}, [x^1,\ldots, x^{n-1}, z_k]]\wdots z_{n-2}\ot v + \\
& \hspace{1cm} \sum_{k=1}^{n-2}\, z_1\wdots [x^1,\ldots, x^{n-1}, z_k]\wdots z_{n-2}\ot \mu(y^1,\ldots,y^{n-1})(v) + \\
& \hspace{2cm} \sum_{p=1}^{n-2}\,z_1\wdots [y^1,\ldots, y^{n-1}, z_p] \wdots z_{n-2}\ot \mu(x^1,\ldots,x^{n-1})(v) + \\
&\hspace{3cm}   z_1\wdots z_{n-2}\ot \mu(y^1,\ldots,y^{n-1})\mu(x^1,\ldots,x^{n-1})(v).
\end{align*} 
Accordingly, it follows at once that
\begin{align*}
& x\rt (y\rt (z_1\wdots z_{n-2}\ot v) ) - y\rt (x\rt (z_1\wdots z_{n-2}\ot v) ) = \\
& \sum_{k=1}^{n-2}\, z_1\wdots [x^1,\ldots, x^{n-1}, [y^1,\ldots, y^{n-1}, z_k]]\wdots z_{n-2}\ot v - \\
& \hspace{1cm}\sum_{k=1}^{n-2}\, z_1\wdots [y^1,\ldots, y^{n-1}, [x^1,\ldots, x^{n-1}, z_k]]\wdots z_{n-2}\ot v + \\
& z_1\wdots z_{n-2}\ot \mu(x^1,\ldots,x^{n-1})\mu(y^1,\ldots,y^{n-1})(v) - \\
& \hspace{2cm} z_1\wdots z_{n-2}\ot \mu(y^1,\ldots,y^{n-1})\mu(x^1,\ldots,x^{n-1})(v),
\end{align*}
from which, and \eqref{rep-n-Lie-I}, we deduce \eqref{rep-1}. 

As for \eqref{rep-2} we compute, on one hand,
\begin{align*}
& (x\rt (z_1\wdots z_{n-2}\ot v) )\lt y = \\
& \sum_{k=1}^{n-2}\sum_{p=1}^{n-1}\,(-1)^{n+p}\,y^1\wdots \widehat{y^p}\wdots y^{n-1}\ot\mu(z_1,\ldots, [x^1,\ldots, x^{n-1}, z_k],\ldots, z_{n-2},y^p)(v) +\\
& \sum_{p=1}^{n-1}\,(-1)^{n+p}\,y^1\wdots \widehat{y^p}\wdots y^{n-1}\ot\mu(z_1,\ldots, z_{n-2},y^p)\mu(x^1,\ldots, x^{n-1})(v), 
\end{align*}
and on the other hand
\begin{align*}
& x\rt ((z_1\wdots z_{n-2}\ot v) \lt y) = \\
& \sum_{p=1}^{n-1}\,(-1)^{n+p}\, x\rt \Big(y^1\wdots \widehat{y^p}\wdots y^{n-1}\ot \mu(z_1,\ldots, z_{n-2}, y^p)(v)\Big) = \\
& \sum_{p=1}^{n-1}\underset{k\neq p}{\sum_{k=1}^{n-1}}\,(-1)^{n+p}\, y^1\wdots [x^1,\ldots,x^{n-1},y^k] \wdots \widehat{y^p}\wdots y^{n-1}\ot \mu(z_1,\ldots, z_{n-2}, y^p)(v) + \\
& \sum_{p=1}^{n-1}\,(-1)^{n+p}\,y^1\wdots \widehat{y^p}\wdots y^{n-1}\ot \mu(x^1,\ldots, x^{n-1})\mu(z_1,\ldots, z_{n-2}, y^p)(v).
\end{align*}
We thus conclude \eqref{rep-2}, in view of \eqref{rep-n-Lie-I}. 

Finally, we proceed onto \eqref{rep-3}. To this end, it suffices to compute
\begin{align*}
& ((z_1\wdots z_{n-2}\ot v) \lt x) \lt y = \\
& \sum_{k=1}^{n-1}\,(-1)^{n+k}\, \Big(x^1\wdots \widehat{x^k}\wdots x^{n-1}\ot \mu(z_1,\ldots, z_{n-2}, x^k)(v)\Big) \lt y = \\
& \sum_{k=1}^{n-1}\sum_{p=1}^{n-1}\,(-1)^{k+p}\, y^1\wdots \widehat{y^p}\wdots y^{n-1}\ot\mu(x^1,\ldots, \widehat{x^k},\ldots, x^{n-1},y^p) \mu(z_1,\ldots, z_{n-2}, x^k)(v),
\end{align*}
since, in view of \eqref{rep-n-Lie-I}, 
\begin{align*}
& (x\rt (z_1\wdots z_{n-2}\ot v) )\lt y = \\
& \sum_{p=1}^{n-1}\,(-1)^{n+p}\,y^1\wdots \widehat{y^p}\wdots y^{n-1}\ot \mu(x^1,\ldots, x^{n-1})\mu(z_1,\ldots, z_{n-2},y^p)(v) -\\
& \sum_{p=1}^{n-1}\,(-1)^{n+p}\,y^1\wdots \widehat{y^p}\wdots y^{n-1}\ot\mu(z_1,\ldots, z_{n-2},[x^1,\ldots, x^{n-1},y^p])(v) = \\
& \sum_{p=1}^{n-1}\,(-1)^{n+p}\,y^1\wdots \widehat{y^p}\wdots y^{n-1}\ot \mu(x^1,\ldots, x^{n-1})\mu(z_1,\ldots, z_{n-2},y^p)(v) + \\
& \sum_{p=1}^{n-1}\,(-1)^{p+1}\,y^1\wdots \widehat{y^p}\wdots y^{n-1}\ot\mu([x^1,\ldots, x^{n-1},y^p],z_1,\ldots, z_{n-2})(v).
\end{align*}
The condition \eqref{rep-3}, then, follows from \eqref{rep-n-Lie-II}.
\end{proof}

Now, the relation between the cohomology of a $n$-Lie algebra, and its associated Leibniz algebra of fundamental objects is given below.

\begin{proposition}\label{prop-D-n-Lie-map-of-complexes}
Given a $n$-Lie algebra $\C{L}$, together with a representation $(V,\mu)$, the diagram 
\[
\xymatrix{
C^m(\C{L}, V) \ar[d]_\d \ar[r]^{\D^m\,\,\,\,\,\,\,\,\,\,\,\,\,\,\,}  & CL^m(\C{L}_{n-1}, \C{L}_{n-2}\ot V) \ar[d]^d \\
C^{m+1}(\C{L}, V) \ar[r]^{\D^{m+1}\,\,\,\,\,\,\,\,\,\,\,\,\,\,\,} & CL^{m+1}(\C{L}_{n-1}, \C{L}_{n-2}\ot V)
}
\]
is commutative for any $m\geq 0$, where
\begin{align*}
& \D^m:C^m(\C{L},V) \to CL^m(\C{L}_{n-1},\C{L}_{n-2}\ot V), \qquad f \mapsto \D^m(f),\\
& \D^m(f) (x_1,\ldots,x_m):= \sum_{k=1}^{n-1}\,(-1)^{k}\,x_m^1\wdots \widehat{x_m^k}\wdots x_m^{n-1}\ot f(x_1,\ldots, x_{m-1},x_m^k), \qquad m\geq 1,\\
& \D^0:=-\Id:\C{L}_{n-2}\ot V \to \C{L}_{n-2}\ot V.
\end{align*}
\end{proposition}

\begin{proof}
On the one hand we have
\begin{align*}
& \D^{m+1}(\d f) (x_1,\ldots,x_{m+1}) = \sum_{k=1}^{n-1}\,(-1)^{k}\,x_{m+1}^1\wdots \widehat{x_{m+1}^k}\wdots x_{m+1}^{n-1}\ot \d f(x_1,\ldots, x_m,x_{m+1}^k) = \\
& \sum_{1\leq i < j \leq m}\sum_{k=1}^{n-1}\,(-1)^{i+k}\,x_{m+1}^1\wdots \widehat{x_{m+1}^k}\wdots x_{m+1}^{n-1}\ot  f(x_1,\ldots,\widehat{x_i},\ldots,  [x_i,x_j], \ldots, x_m,x_{m+1}^k) + \\
& \sum_{i=1}^{m}\sum_{k=1}^{n-1}\,(-1)^{i+k}\,x_{m+1}^1\wdots \widehat{x_{m+1}^k}\wdots x_{m+1}^{n-1}\ot  f(x_1,\ldots,\widehat{x_i},\ldots, x_m,[x_i^1,\ldots,x_i^{n-1},x_{m+1}^k]) + \\
& \sum_{i=1}^{m}\sum_{k=1}^{n-1}\,(-1)^{i+1+k}\,x_{m+1}^1\wdots \widehat{x_{m+1}^k}\wdots x_{m+1}^{n-1}\ot \mu(x_i^1,\ldots,x_i^{n-1})\Big( f(x_1,\ldots,\widehat{x_i},\ldots, x_m,x_{m+1}^k)\Big) + \\
& \sum_{i=1}^{n-1}\sum_{k=1}^{n-1}\,(-1)^{n-1+m+i+k}\,x_{m+1}^1\wdots \widehat{x_{m+1}^k}\wedge \cdots \\
& \hspace{2cm} \cdots \wedge x_{m+1}^{n-1}\ot \mu(x_{m}^1\wdots \widehat{x_{m}^i}\wdots x_{m}^{n-1},x_{m+1}^k) \Big(f(x_1,\ldots, x_{m-1},x_{m}^i)\Big),
\end{align*}
while on the other hand
\begin{align*}
& d(\D^m f) (x_1,\ldots,x_{m+1}) = \sum_{1\leq i < j\leq m+1}\,(-1)^i\,(\D^m f)(x_1,\ldots,\widehat{x_i},\ldots,  [x_i,x_j], \ldots, x_{m+1}) +\\
& \sum_{k=1}^{m}\,(-1)^{k+1}\,x_k\rt (\D^m f)(x_1,\ldots,\widehat{x_k},\ldots,x_{m+1}) +  (-1)^{m+1}\,(\D^m f)(x_1,\ldots,x_{m}) \lt x_{m+1} = \\
& \sum_{1\leq i < j\leq m}\,(-1)^i\,(\D^m f)(x_1,\ldots,\widehat{x_i},\ldots,  [x_i,x_j], \ldots, x_{m+1}) +\\
& \sum_{1\leq i \leq m}\,(-1)^i\,(\D^m f)(x_1,\ldots,\widehat{x_i},\ldots,  x_m,[x_i,x_{m+1}]) +\\
& \sum_{k=1}^{m}\,(-1)^{k+1}\,x_k\rt (\D^m f)(x_1,\ldots,\widehat{x}_k,\ldots,x_{m+1}) +  (-1)^{m+1}\,(\D^m f)(x_1,\ldots,x_{m}) \lt x_{m+1}.
\end{align*}
Accordingly,
\begin{align*}
& d(\D^m f) (x_1,\ldots,x_{m+1}) = \\
& \sum_{1\leq i < j\leq m}\sum_{k=1}^{n-1}\,(-1)^{i+k}\,x_{m+1}^1\wdots \widehat{x_{m+1}^k} \wdots x_{m+1}^{n-1}\ot f(x_1,\ldots,\widehat{x_i},\ldots,  [x_i,x_j], \ldots, x_{m+1}^k) +\\
& \sum_{p=1}^{n-1}\underset{k\neq p}{\sum_{k=1}^{n-1}}\sum_{i=1}^m\,(-1)^{k+i}\, x_{m+1}^1\wdots \widehat{x_{m+1}^k}\wdots [x_i^1\ldots,x_i^{n-1},x_{m+1}^p]\wedge \cdots \\
& \hspace{2cm} \cdots \wedge x_{m+1}^{n-1}\ot f(x_1,\ldots,\widehat{x_i},\ldots,  x_m,x_{m+1}^k) +\\
& \sum_{p=1}^{n-1}\sum_{i=1}^m\,(-1)^{p+i}\, x_{m+1}^1\wdots \widehat{x_{m+1}^p} \wdots  x_{m+1}^{n-1}\ot f(x_1,\ldots,\widehat{x_i},\ldots,  x_m,[x_i^1\ldots,x_i^{n-1},x_{m+1}^p]) + \\
& \sum_{i=1}^{m}\,(-1)^{i+1}\,x_i\rt (\D^m f)(x_1,\ldots,\widehat{x}_i,\ldots,x_{m+1}) +  (-1)^{m+1}\,(\D^m f)(x_1,\ldots,x_{m}) \lt x_{m+1}.
\end{align*}
Next, we note that
\begin{align*}
& \sum_{i=1}^{m}\,(-1)^{i+1}\,x_i\rt (\D^m f)(x_1,\ldots,\widehat{x}_i,\ldots,x_{m+1}) = \\
& \sum_{k=1}^{n-1}\sum_{i=1}^{m}\,(-1)^{i+1+k}\,x_i\rt \Big(x_{m+1}^1 \wdots \widehat{x_{m+1}^k} \wdots x_{m+1}^{n-1}\ot f(x_1,\ldots,\widehat{x}_i,\ldots,x_{m+1}^k)\Big) = \\
& \underset{p\neq k}{\sum_{p=1}^{n-1}}\sum_{k=1}^{n-1}\sum_{i=1}^{m}\,(-1)^{i+1+k}\, x_{m+1}^1\wdots \widehat{x_{m+1}^k}\wdots [x_i^1\ldots,x_i^{n-1},x_{m+1}^p]\wedge \cdots \\
& \hspace{2cm} \cdots \wedge x_{m+1}^{n-1}\ot f(x_1,\ldots,\widehat{x_i},\ldots,  x_m,x_{m+1}^k) + \\
&  \sum_{k=1}^{n-1}\sum_{i=1}^{m}\,(-1)^{i+1+k}\, x_{m+1}^1 \wdots \widehat{x_{m+1}^k} \wdots x_{m+1}^{n-1}\ot \mu(x_i^1,\ldots,x_i^{n-1})\Big(f(x_1,\ldots,\widehat{x}_i,\ldots,x_{m+1}^k)\Big),
\end{align*}
and that
\begin{align*}
& (-1)^{m+1}\,(\D^m f)(x_1,\ldots,x_{m}) \lt x_{m+1} = \\
& \sum_{k=1}^{n-1}\,(-1)^{m+1+k}\,\Big(x_m^1\wdots \widehat{x_m^k}\wdots x_m^{n-1}\ot f(x_1,\ldots, x_{m-1},x_m^k)\Big)\lt x_{m+1} = \\
& \sum_{i=1}^{n-1}\sum_{k=1}^{n-1}\,(-1)^{n+i+m+1+k}\,x_{m+1}^1\wdots \widehat{x_{m+1}^i}\wedge \cdots \\
& \hspace{2cm} \cdots \wedge x_{m+1}^{n-1}\ot \mu(x_{m}^1\wdots \widehat{x_{m}^k}\wdots x_{m}^{n-1},x_{m+1}^i) \Big(f(x_1,\ldots, x_{m-1},x_{m}^k)\Big).
\end{align*}
The result, then, follows.
\end{proof}

The homology $H(C(\C{L},V), \delta)$ of the differential complex \eqref{n-Lie-diff-comp} is called the \emph{$n$-Lie algebra cohomology} of $\C{L}$, with coefficients in $V$, and is denoted by $H^\ast(\C{L},V)$. We shall denote by $Z^m(\C{L},V)$ the space of $m$-cocycles, and by $B^m(\C{L},V)$ the space of $m$-coboundaries.

\subsection{$H^2(\C{L},V)$ and abelian extensions}\label{subsect-ab-ext}~

The cohomology theory for $n$-Lie algebras is qualified to classify the abelian extensions of a $n$-Lie algebra with a representation space of it. More precisely \cite[Thm. 5.7]{Zhan14}, see also \cite{LiuMakhShen17}, extends to $n$-Lie algebras as follows.

\begin{proposition}\label{prop-ab-ext}
Given a $n$-Lie algebra $\C{L}$, a representation $(V,\mu)$ of $\C{L}$, and any $f\in C^2(\C{L},V)$, the space $V {}_f\rtimes \C{L} := V\oplus \C{L}$ is a $n$-Lie algebra through
\begin{equation}\label{2-cocycle-ext-for-n-Lie}
[v_1+x_1,\ldots,v_n+x_n] := \sum_{k=1}^{n}\,(-1)^{k+1}\,\mu(x_1,\ldots,\widehat{x}_k,\ldots,x_n)(v_k) + (-1)^{n+1}f(x_1,\ldots,x_n) + (-1)^{n+1}[x_1,\ldots,x_n]
\end{equation}
if and only if $f \in Z^2(\C{L},V)$. 
\end{proposition}

\begin{proof}
Let us begin with the fundamental identity \eqref{fund-ID-n-Lie}, that is,
\begin{align}\label{FI-n-Lie}
\begin{split}
& [v_1+x_1,\ldots, v_{n-1}+x_{n-1}, [v_n+x_n,\ldots,v_{2n-1}+x_{2n-1}]] =\\ 
& \sum_{k=0}^{n-1}\, [v_n+x_n, \ldots,  [v_1+x_1,\ldots,v_{n-1}+x_{n-1}, v_{n+k}+x_{n+k}],\ldots, v_{2n-1}+x_{2n-1}],
\end{split}
\end{align}
where, on the left hand side we have
\begin{align*}
& [v_1+x_1,\ldots, v_{n-1}+x_{n-1}, [v_n+x_n,\ldots,v_{2n-1}+x_{2n-1}]] =\\ 
&  [v_1+x_1,\ldots, v_{n-1}+x_{n-1}, \\
&  \sum_{p=0}^{n-1}  \,(-1)^p\,\mu(x_n,\ldots,\widehat{x_{n+p}},\ldots,x_{2n-1})(v_{n+p}) + (-1)^{n+1} f(x_n,\ldots,x_{2n-1}) + (-1)^{n+1}[x_n,\ldots, x_{2n-1}]] = \\
& \sum_{r=1}^{n-1}  \,(-1)^{r+n}\,\mu(x_1,\ldots, \widehat{x}_r,\ldots,[x_n,\ldots, x_{2n-1}])(v_r)  + \\
& \sum_{p=0}^{n-1}  \,(-1)^{p+n+1}\,\mu(x_1,\ldots, x_{n-1})\mu(x_n,\ldots,\widehat{x_{n+p}},\ldots,x_{2n-1})(v_{n+p}) + \\
& \mu(x_1,\ldots, x_{n-1})\Big(f(x_n,\ldots,x_{2n-1})\Big) +   f(x_1,\ldots,x_{n-1},[x_n,\ldots, x_{2n-1}]) + [x_1,\ldots,x_{n-1},[x_n,\ldots, x_{2n-1}]],
\end{align*}
while on the right hand side
\begin{align*}
& \sum_{k=0}^{n-1}\, [v_n+x_n, \ldots,  [v_1+x_1,\ldots,v_{n-1}+x_{n-1}, v_{n+k}+x_{n+k}],\ldots, v_{2n-1}+x_{2n-1}] = \\
& \sum_{k=0}^{n-1}\, [v_n+x_n, \ldots  \\
& \ldots,\sum_{r=1}^{n-1}\,(-1)^{r+1}\,\mu(x_1,\ldots,\widehat{x}_r,\ldots,x_{n-1},x_{n+k})(v_r) + (-1)^{n+1}\mu(x_1,\ldots,x_{n-1})(v_{n+k}) +\\
& \hspace{3cm}  (-1)^{n+1}f(x_1,\ldots,x_{n-1},x_{n+k}) + (-1)^{n+1}[x_1,\ldots,x_{n-1},x_{n+k}], \ldots \\
& \hspace{2cm} \ldots, v_{2n-1}+x_{2n-1}] = \\
& \underset{p\neq k}{\sum_{p=0}^{n-1}}\sum_{k=0}^{n-1}\,(-1)^{p+n+1}\, \mu(x_n,\ldots,\widehat{x_{n+p}},\ldots,[x_1,\ldots,x_{n-1},x_{n+k}],\ldots,x_{2n-1})(v_{n+p}) + \\
&\sum_{r=1}^{n-1}\sum_{k=0}^{n-1}\,(-1)^{k+r+1}\, \mu(x_n,\ldots,\widehat{x_{n+k}},\ldots,x_{2n-1})\mu(x_1,\ldots,\widehat{x}_r,\ldots,x_{n-1},x_{n+k})(v_r) + \\
& \sum_{k=0}^{n-1}\,(-1)^{k+n+1}\, \mu(x_n,\ldots,\widehat{x_{n+k}},\ldots,x_{2n-1})\mu(x_1,\ldots,\ldots,x_{n-1})(v_{n+k}) + \\
& \sum_{k=0}^{n-1}\,(-1)^{k+n+1}\, \mu(x_n,\ldots,\widehat{x_{n+k}},\ldots,x_{2n-1}) \Big(f(x_1,\ldots,x_{n-1},x_{n+k})\Big) + \\
& \sum_{k=0}^{n-1}\, f(x_n,\ldots,[x_1,\ldots,x_{n-1},x_{n+k}],\ldots,x_{2n-1}) +  \sum_{k=0}^{n-1}\, [x_n,\ldots,[x_1,\ldots,x_{n-1},x_{n+k}],\ldots,x_{2n-1}].
\end{align*}
Accordingly, in view of \eqref{rep-n-Lie-I} and \eqref{rep-n-Lie-II}, the fundamental identity \eqref{FI-n-Lie} holds if and only if 
\begin{align*}
& \mu(x_1,\ldots, x_{n-1})\Big(f(x_n,\ldots,x_{2n-1})\Big) +   f(x_1,\ldots,x_{n-1},[x_n,\ldots, x_{2n-1}]) = \\
& \sum_{k=0}^{n-1}\,(-1)^{k+n+1}\, \mu(x_n,\ldots,\widehat{x_{n+k}},\ldots,x_{2n-1}) \Big(f(x_1,\ldots,x_{n-1},x_{n+k})\Big) + \\
& \sum_{k=0}^{n-1}\, f(x_n,\ldots,[x_1,\ldots,x_{n-1},x_{n+k}],\ldots,x_{2n-1}), 
\end{align*}
which is equivalent to
\[
(\d f)(x_1,\ldots,x_{n-1},x_n,\ldots, x_{2n-2},x_{2n-1}) = 0
\]
for any $x_1,\ldots,x_{2n-1} \in \C{L}$.
\end{proof}

This $n$-Lie algebra $V {}_f\rtimes \C{L} $ is called the \emph{abelian extension} of $\C{L}$ by $V$ via $f\in Z^2(\C{L},V)$.

\begin{proposition}\label{prop-ab-ext-II}
Given a $n$-Lie algebra $\C{L}$, a representation $(V,\mu)$ of $\C{L}$, and $f,g\in Z^2(\C{L},V)$, we have $V {}_f\rtimes \C{L} \cong V {}_g\rtimes \C{L}$ if and only if $f-g \in B^2(\C{L},V)$. 
\end{proposition}

\begin{proof}
Let $H:V {}_f\rtimes \C{L} \cong V {}_g\rtimes \C{L}$ be the $n$-Lie algebra isomorphism, which may be given by
\[
H(v+x) = v + h(x) + x, 
\]
for some $h\in C^1(\C{L},V)$, where $v\in V$ and $x\in \C{L}$. Then, it follows from
\[
H([v_1+x_1,\ldots,v_n+x_n]) = [H(v_1+x_1),\ldots,H(v_n+x_n)]
\]
that
\begin{align*}
& \sum_{k=1}^n\,(-1)^{k+1}\,\mu(x_1,\ldots,\widehat{x}_k,\ldots,x_n)(v_k) + (-1)^{n+1}f(x_1,\ldots,x_n) + \\
& \hspace{4cm} (-1)^{n+1}h([x_1,\ldots,x_n]) + (-1)^{n+1}[x_1,\ldots,x_n] = \\
& \sum_{k=1}^n\,(-1)^{k+1}\,\mu(x_1,\ldots,\widehat{x}_k,\ldots,x_n)(v_k + h(x_k)) + (-1)^{n+1}g(x_1,\ldots,x_n) +(-1)^{n+1}[x_1,\ldots,x_n], 
\end{align*}
end hence that
\begin{align*}
&  (-1)^{n+1}f(x_1,\ldots,x_n) + (-1)^{n+1}h([x_1,\ldots,x_n]) = \\
& \sum_{k=1}^n\,(-1)^{k+1}\,\mu(x_1,\ldots,\widehat{x}_k,\ldots,x_n)(h(x_k)) + (-1)^{n+1}g(x_1,\ldots,x_n), 
\end{align*}
which is equivalent to 
\[
f(x_1,\ldots,x_n) - g(x_1,\ldots,x_n) = (\d h)(x_1,\ldots,x_n),
\]
for any $x_1,\ldots,x_n \in \C{L}$.
\end{proof}

As a result, we conclude that the abelian extensions of a $n$-Lie algebra $\C{L}$ by $V$ are in one-to-one correspondence with $H^2(\C{L},V)$.

\subsection{$H^2(\C{L},\C{L})$ and infinitesimal deformations}\label{subsect-inf-deform}~

Let us next discuss briefly the infinitesimal deformations of a $n$-Lie algebra. More precisely, we shall upgrade \cite[Thm. 3.1]{Zhan14} to $n$-Lie algebras. 

In accordance with the terminology for Lie algebras, let us define a \emph{deformation} of a $n$-Lie algebra $\C{L}$ to be a smooth map 
\[
h:\C{L}^{\times\,n} \times \B{R}\to \C{L},
\]
so that $h(x_1,\ldots,x_n,0):=[x_1,\ldots,x_n]$ on $\C{L}$, and that $[x_1,\ldots,x_n]_t:=h(x_1,\ldots,x_n,t)$ determines a $n$-Lie algebra structure on $\C{L}$, for any $t\in \B{R}$. Similarly, an \emph{infinitesimal deformation} of $\C{L}$ may be defined to be 
\[
\eta:\C{L}^{\times\,n} \to \C{L}, \qquad \eta(x_1,\ldots,x_n):=\dt h(x_1,\ldots,x_n,t),
\]
which induces, by differentiating, that $\eta:\C{L}_n \to \C{L}$. In other words, $\eta\in C^2(\C{L},\C{L})$. Furthermore, the derivative of the fundamental identity 
\[
[x_1,\ldots, x_{n-1}, [x_n,\ldots,x_{2n-1}]_t]_t = \sum_{k=0}^{n-1}\, [x_n, \ldots,  [x_1,\ldots,x_{n-1}, x_{n+k}]_t,\ldots, x_{2n-1}]_t
\]
yields
\begin{align*}
& \eta(x_1,\ldots, x_{n-1}, [x_n,\ldots,x_{2n-1}]) + [x_1,\ldots, x_{n-1}, \eta(x_n,\ldots,x_{2n-1})] = \\
& \sum_{k=0}^{n-1}\, \eta(x_n, \ldots,  [x_1,\ldots,x_{n-1}, x_{n+k}],\ldots, x_{2n-1}) + \\
& \hspace{3cm} \sum_{k=0}^{n-1}\, [x_n, \ldots,  \eta(x_1,\ldots,x_{n-1}, x_{n+k}),\ldots, x_{2n-1}]
\end{align*}
that is 
\[
(\d\eta)(x_1,\ldots x_{n-1},x_n,\ldots,x_{2n-2},x_{2n-1}) = 0,
\]
or equivalently $\eta\in Z^2(\C{L},\C{L})$.

Let us further call two deformations $h_1:\C{L}^{\times\,n} \times \B{R}\to \C{L}$ and $h_2:\C{L}^{\times\,n} \times \B{R}\to \C{L}$ to be equivalent, if there is $u:\C{L}\times \B{R} \to\C{L}$ such that $u(x,0) = x$, and that
\begin{equation}\label{deform-n-Lie}
u(h_1(x_1,\ldots,x_n,t),t) = h_2(u(x_1,t),\ldots,u(x_n,t),t)
\end{equation}
for any $x_1,\ldots,x_n \in \C{L}$, and any $t\in \B{R}$. Differentiating \eqref{deform-n-Lie} we see that the corresponding infinitesimal deformations
\[
\eta_1(x_1,\ldots,x_n) := \dt h_1(x_1,\ldots,x_n,t), \qquad \eta_2(x_1,\ldots,x_n) := \dt h_2(x_1,\ldots,x_n,t)
\]
satisfy
\begin{align*}
& g([x_1,\ldots,x_n]) + \eta_1(x_1,\ldots,x_n) = \eta_2(x_1,\ldots,x_n) + \sum_{k=1}^n\, [x_1,\ldots,g(x_k),\ldots,x_n],
\end{align*}
that is,
\[
\eta_1(x_1,\ldots,x_{n-1},x_n) - \eta_2(x_1,\ldots,x_{n-1},x_n) = (\d g)(x_1,\ldots,x_{n-1},x_n),
\]
for 
\[
g:\C{L}\to \C{L}, \qquad g(x):=\dt u(x,t).
\]
Equivalently $\eta_1 - \eta_2 \in B^2(\C{L},\C{L})$. We thus observe that the infinitesimal deformations of $\C{L}$ are in one-to-one correspondence with $H^2(\C{L},\C{L})$.

\subsection{$\C{H}^1(\C{L},g\ell(\C{L}))$ and generalized derivation extensions}\label{subsect-gen-der-n-Lie}~

The present subsection accommodates the main result of the paper. Namely, we shall now show that a generalized derivation on a $n$-Lie algebra may be realized as a 1-cocycle, albeit in an unusual $n$-Lie algebra cohomology. Accordingly, we shall begin with the presentation of this alternate cohomology complex for $n$-Lie algebras.

Let $\C{L}$ be a $n$-Lie algebra, and let $(V,\mu)$ be a representation of $\C{L}$. Then, setting $x_i := x_i^1\wdots x_i^{n-1} \in \C{L}_{n-1}$, and $y := y^1\wdots y^{n-2}\in\C{L}_{n-2} $, 
\begin{equation}\label{n-Lie-diff-comp-II}
\C{C}(\C{L},V) := \bigoplus_{m\geq 0} \,\C{C}^m(\C{L},V), \qquad \C{C}^m(\C{L},V):= \Hom\Big(\C{L}_{n-1}^{\wedge\,(m-1)}\wedge \C{L}_{n-2}, V\Big), \quad m\geq 1,
\end{equation}
where $\C{C}^0(\C{L},V):= \C{L}\ot V$, is a differential complex via
\begin{align}\label{n-Lie-diff-II}
\begin{split}
& \delta:\C{C}^0(\C{L},V) \to \C{C}^1(\C{L},V), \qquad \d(z\ot v) (y) := \mu(z,y^1,\ldots,y^{n-2})(v), \\
& \delta:\C{C}^m(\C{L},V) \to \C{C}^{m+1}(\C{L},V), \qquad m\geq 1\\
& \delta f(x_1,\ldots, x_m,y)  := \\
& \sum_{1\leq i < j \leq m}\, (-1)^i\, f(x_1,\ldots,\widehat{x_i},\ldots,  [x_i,x_j], \ldots, x_m,y) +\\
 & \hspace{3cm}  \sum_{i=1}^{m}\, (-1)^i\, f(x_1,\ldots,\widehat{x_i},\ldots,x_m,[x_i,y]) +\\
 &\hspace{3cm} \sum_{i=1}^{m}\, (-1)^{i+1} \,\mu(x_i^1,\ldots,x_i^{n-1})\Big( f(x_1,\ldots,\widehat{x_i},\ldots,x_m,y)\Big) +\\
 & \hspace{3cm}  \sum_{i=1}^{n-1}\,(-1)^{m+i+1}\, \mu(x_m^i,y^1,\ldots,y^{n-2}) \Big(f(x_1,\ldots,x_{m-1}, X_m^i)\Big),
 \end{split}
\end{align}
for any $f\in \C{C}^m(\C{L},V)$, where $X_m^i:= x_m^1\wdots\widehat{x_m^i}\wdots x_m^{n-1} \in \C{L}_{n-2}$,
\[
[x_i,x_j]:= \sum_{k=1}^{n-1}\,x_j^1\wdots [x_i^1,\ldots, x_i^{n-1},x_j^k] \wdots x_j^{n-1} \in \C{L}_{n-1},
\]
and
\[
[x_i,y]:= \sum_{k=1}^{n-2}\,y^1\wdots [x_i^1,\ldots,x_i^{n-1},y^k] \wdots y^{n-2} \in \C{L}_{n-2}.
\]

Towards \eqref{n-Lie-diff-II} being indeed a differential, we first note from Proposition \ref{prop-L-n-2-ot-V-rep-L-n-1} that $\C{L}\ot V$ is a representation through 
\begin{align}
& x\rt (z\ot v) :=  [x^1,\ldots, x^{n-1}, z]\ot  v + z\ot \mu(x^1,\ldots, x^{n-1})(v), \label{rep-Leibniz-n-Lie-I-II}\\
& (z\ot v) \lt x := \sum_{k=1}^{n-1}\,(-1)^{k}\, x^k\wedge \mu(z, X^k)(v), \label{rep-Leibniz-n-Lie-II-II}
\end{align}
where $x:=x^1\wdots x^{n-1}\in \C{L}_{n-1}$, and $X^k:=x^1\wdots \widehat{x^k}\wdots x^{n-1}\in \C{L}_{n-2}$. Then we observe the following analogue of Proposition \ref{prop-D-n-Lie-map-of-complexes}.

\begin{proposition}
Given a $n$-Lie algebra $\C{L}$, together with a representation $(V,\mu)$, the diagram 
\[
\xymatrix{
\C{C}^m(\C{L}, V) \ar[d]_\d \ar[r]^{\Theta^m\,\,\,\,\,\,\,\,\,\,\,\,\,\,\,}  & CL^m(\C{L}_{n-1}, \C{L}\ot V) \ar[d]^d \\
\C{C}^{m+1}(\C{L}, V) \ar[r]^{\Theta^{m+1}\,\,\,\,\,\,\,\,\,\,\,\,\,\,\,} & CL^{m+1}(\C{L}_{n-1}, \C{L}\ot V)
}
\]
is commutative for any $m\geq 0$, where
\begin{align*}
& \Theta^m:\C{C}^m(\C{L},V) \to CL^m(\C{L}_{n-1},\C{L}\ot V), \qquad f \mapsto \Theta^m(f),\\
& \Theta^m(f) (x_1,\ldots,x_m):= \sum_{k=1}^{n-1}\,(-1)^{k+1}\, x_m^k\ot f(x_1,\ldots, x_{m-1},X_m^k), \qquad m\geq 1,\\
& \Theta^0:=-\Id:\C{L}\ot V \to \C{L}\ot V,
\end{align*}
and $X_m^k:=x_m^1\wdots \widehat{x_m^k}\wdots x_m^{n-1}\in \C{L}_{n-2}$
\end{proposition}

\begin{proof}
On the one hand we have
\begin{align*}
& \Theta^{m+1}(\d f) (x_1,\ldots,x_{m+1}) =  \sum_{k=1}^{n-1}\,(-1)^{k+1}\, x_{m+1}^k\ot (\d f)(x_1,\ldots, x_m,X_{m+1}^k) = \\
& \sum_{k=1}^{n-1}\sum_{1\leq i < j\leq m}\,(-1)^{i+k+1}\, x_{m+1}^k\ot f(x_1,\ldots, \widehat{x_i}, \ldots, [x_i,x_j],\ldots, x_m,X_{m+1}^k) + \\
& \sum_{k=1}^{n-1}\sum_{i=1}^{m}\,(-1)^{i+k+1}\, x_{m+1}^k\ot f(x_1,\ldots, \widehat{x_i}, \ldots, x_m,[x_i,X_{m+1}^k]) + \\
& \sum_{k=1}^{n-1}\sum_{i=1}^{m}\,(-1)^{i+k}\, x_{m+1}^k\ot \mu(x_i^1,\ldots,x_i^{n-1})\Big(f(x_1,\ldots, \widehat{x_i}, \ldots, x_m,X_{m+1}^k)\Big) + \\
& \sum_{k=1}^{n-1}\sum_{i=1}^{n-1}\,(-1)^{i+k+m}\, x_{m+1}^k\ot \mu(x_m^i,X_{m+1}^k)\Big(f(x_1,\ldots, x_m,X_{m}^i)\Big),
\end{align*}
while on the other hand,
\begin{align*}
& d(\Theta^{m} (f)) (x_1,\ldots,x_{m+1}) = \\
& \sum_{1\leq i < j \leq m+1}\,(-1)^i\,(\Theta^{m} (f)) (x_1,\ldots,\widehat{x_i}, \ldots, [x_i,x_j],\ldots, x_{m+1}) + \\
& \sum_{i=1}^m\,(-1)^{i+1}\, x_i\rt \Big((\Theta^{m} (f))(x_1,\ldots,\widehat{x_i}, \ldots, x_{m+1})\Big) + (-1)^{m+1} \, \Big((\Theta^{m} (f))(x_1,\ldots, x_{m})\Big) \lt x_{m+1},
\end{align*}
where on the latter, we note that
\begin{align*}
& \sum_{1\leq i < j \leq m+1}\,(-1)^i\,(\Theta^{m} (f)) (x_1,\ldots,\widehat{x_i}, \ldots, [x_i,x_j],\ldots, x_{m+1}) = \\
& \sum_{1\leq i < j \leq m}\,(-1)^i\,(\Theta^{m} (f)) (x_1,\ldots,\widehat{x_i}, \ldots, [x_i,x_j],\ldots, x_{m+1}) + \\
& \sum_{i=1}^m\,(-1)^i\,(\Theta^{m} (f)) (x_1,\ldots,\widehat{x_i},\ldots, [x_i,x_{m+1}]) = \\
& \sum_{k=1}^{n-1}\sum_{1\leq i < j\leq m}\,(-1)^{i+k+1}\, x_{m+1}^k\ot f(x_1,\ldots, \widehat{x_i}, \ldots, [x_i,x_j],\ldots, x_m,X_{m+1}^k) +\\
& \sum_{k=1}^{n-1}\sum_{i=1}^m\,(-1)^{i+k+1}\,x^k_{m+1}\ot f (x_1,\ldots,\widehat{x_i},\ldots, [x_i,X_{m+1}^k]) + \\
& \sum_{k=1}^{n-1}\sum_{i=1}^m\,(-1)^{i+k+1}\,[x_i^1,\ldots, x_i^{n-1},x^k_{m+1}]\ot f (x_1,\ldots,\widehat{x_i},\ldots, X_{m+1}^k).
\end{align*}
The result, now, follows from \eqref{rep-Leibniz-n-Lie-I-II} and \eqref{rep-Leibniz-n-Lie-II-II}. More precisely,
\begin{align*}
& \sum_{i=1}^m\,(-1)^{i+1}\, x_i\rt \Big((\Theta^{m} (f))(x_1,\ldots,\widehat{x_i}, \ldots, x_{m+1})\Big) = \\
& \sum_{k=1}^{n-1}\sum_{i=1}^m\,(-1)^{i+k}\,[x_i^1,\ldots, x_i^{n-1},x^k_{m+1}]\ot f (x_1,\ldots,\widehat{x_i},\ldots, X_{m+1}^k) + \\
& \sum_{k=1}^{n-1}\sum_{i=1}^{m}\,(-1)^{i+k}\, x_{m+1}^k\ot \mu(x_i^1,\ldots,x_i^{n-1})\Big(f(x_1,\ldots, \widehat{x_i}, \ldots, x_m,X_{m+1}^k)\Big),
\end{align*}
and
\begin{align*}
& (-1)^{m+1} \, \Big((\Theta^{m} (f))(x_1,\ldots, x_{m})\Big) \lt x_{m+1} = \\
& (-1)^{m+1} \, \bigg(\sum_{k=1}^{n-1}\,(-1)^{k+1}\,x_m^k\ot f(x_1,\ldots, x_{m-1},X_m^k)\bigg) \lt x_{m+1} \\
& \sum_{k=1}^{n-1}\sum_{i=1}^{n-1}\,(-1)^{i+k+m}\, x_{m+1}^i\ot \mu(x_m^k,X_{m+1}^i)\Big(f(x_1,\ldots, x_m,X_{m}^k)\Big).
\end{align*}
\end{proof}

We shall denote the homology $H(\C{C}(\C{L},V), \delta)$ of the differential complex \eqref{n-Lie-diff-comp-II} by $\C{H}^\ast(\C{L},V)$. As is customary, we shall denote by $\C{Z}^m(\C{L},V)$ the space of $m$-cocycles, and by $\C{B}^m(\C{L},V)$ the space of $m$-coboundaries.

\begin{remark}
It worths to note that the differentials \eqref{n-Lie-diff} and \eqref{n-Lie-diff-II}, and also the differential complexes \eqref{n-Lie-diff-comp} and \eqref{n-Lie-diff-comp-II}, coincide (only) for $n=3$.
\end{remark}

\begin{proposition}\label{prop-gen-der-1-cocycle-n-Lie}
Given a $n$-Lie algebra $\C{L}$, let $D:\C{L}_{n-1}\to \C{L}$, with $D^\sharp:\C{L}_{n-2} \to g\ell(\C{L})$ being the mapping in \eqref{D-sharp-derivation}. Then,
$D^\sharp \in \C{Z}^1(\C{L},g\ell(\C{L}))$ if \eqref{gen-der-n-Lie-II} is satisfied.
\end{proposition}

\begin{proof}
We have, for any $x:=x^1\wdots x^{n-1}\in \C{L}_{n-1}$, any $y:=y^1\wdots y^{n-2}\in \C{L}_{n-2}$, and any $z\in \C{L}$, 
\begin{align*}
& (\d D^\sharp)(x,y) (z) = - D^\sharp([x,y])(z) + \Big(\mu(x^1,\ldots,x^{n-1})(D^\sharp(y))\Big)(z) + \\
& \sum_{i=1}^{n-1}\,(-1)^{i} \, \Big(\mu(x^i,y)(D^\sharp(X^i))\Big)(z),
\end{align*}
where $X^i:= x^1\wdots \widehat{x^i}\wdots x^{n-1} \in \C{L}_{n-2}$, and
\[
[x,y] := \sum_{k=1}^{n-2}\,y^1\wdots [x^1,\ldots,x^{n-1},y^k]\wdots y^{n-2} \in \C{L}_{n-2}.
\]
Accordingly,
\begin{align*}
& (\d D^\sharp)(x,y) (z) = -D([x,y],z) - D^\sharp(y)([x,z]) + \sum_{i=1}^{n-1}\,(-1)^{i+1} \, D^\sharp(X^i)([x^i,y,z]) = \\
& -D([x,y],z) - D(y,[x,z]) - \sum_{i=1}^{n-1}\, D(x^1,\ldots,[y,z,x^i], \ldots,x^{n-1}).
\end{align*}
Now, if \eqref{gen-der-n-Lie-II} is satisfied, then
\begin{align*}
& (\d D^\sharp)(x,y) (z) = \\
& -\bigg([x,D(y,z)] + (-1)^n[D(x),y,z]\bigg) -   \bigg([y,z,D(x)] + (-1)^n[D(y,z),x]\bigg) = 0.
\end{align*}
\end{proof}

\section{The Hochschild-Serre spectral sequence for $n$-Lie algebras}\label{sect-Hoch-Serr-spect-seq}~

We shall now introduce, along the lines of \cite{HochSerr53}, a spectral sequence associated to a decreasing filtration.

\subsection{The first page relative to a subalgebra}~

Let $\C{L}$ be a $n$-Lie algebra, and let $\C{K}\subseteq \C{L}$ be a $n$-Lie subalgebra. That is, if $\C{L}$ is a $n$-Lie algebra through $[\,,\ldots,\,]:\C{L}_n \to \C{L}$, then $[\,,\ldots,\,]:\C{K}_n \to \C{K}$.

Let us also record that $C^m(\C{L},V)$ is a representation of $\C{L}_{n-1}$, and hence of $\C{K}_{n-1}$, for any $m\geq 0$, and any representation $(V,\mu)$ of $\C{L}$.

\begin{proposition}
Given any $n$-Lie algebra $\C{L}$, and a representation $(V,\mu)$ of $\C{L}$, the space $C^m(\C{L},V)$ $m$-cohains is a representation space for the Leibniz algebra $\C{L}_{n-1}$ for any $m\geq 0$, through
\begin{align}\label{Leibniz-action-on-complex}
\begin{split}
& (z\rt f)(x_1,\ldots,x_{m-1},y) := \\
& z\rt  f(x_1,\ldots,x_{m-1},y) - \sum_{k=1}^{m-1}\, f(x_1,\ldots, [z,x_k],\ldots, x_{m-1},y) - f(x_1,\ldots,x_{m-1},[z,y]), \\
& f\lt z := - z\rt f, 
\end{split}
\end{align}
for any $x_k:=x_k^1\wdots x_k^{n-1}, z:=z^1\wdots z^{n-1} \in \C{L}_{n-1}$, $1\leq k \leq m-1$, and any $y\in \C{L}$, where
\begin{align*}
& z\rt  f(x_1,\ldots,x_{m-1},y):=\mu(z^1,\ldots,z^{n-1})\Big( f(x_1,\ldots,x_{m-1},y)\Big), \\
& [z,x_k] := \sum_{p=1}^{n-1}\,x_k^1 \wdots [z^1,\ldots,z^{n-1},x_k^p] \wdots x_k^{n-1}, \\
& [z,y]:=[z^1,\ldots,z^{n-1},y],
\end{align*}
\end{proposition}

\begin{proof}
Setting
\[
f_{x_1}(x_2,\ldots,x_{m-1},y) := f(x_1,\ldots,x_{m-1},y),
\]
we see at once that
\[
(z\rt f)(x_1,\ldots, x_{m-1},y) = \Big((z\rt f_{x_1}) - f_{[z,x_1]}\Big)(x_2,\ldots,x_{m-1},y).
\]
The claim, then follows from the induction over the dimension of the cohains.
\end{proof}

See also \cite[Sect. 4]{daletskii1997leibniz}, or \eqref{Lie-derivative}.

Following \cite[Sect. 2]{HochSerr53}, let us next introduce the (decreasing) filtration by 
\[
F_jC^{m}(\C{L},V) := C^{m}(\C{L},V), \quad j\leq 0, \qquad F_jC^{m}(\C{L},V) := 0, \quad j\geq m+1,
\]
and for $1\leq j \leq m$, setting $F_jC^{m}(\C{L},V)$ to be those in $C^{m}(\C{L},V)$ that vanish if $(m-j)(n-1)+1$-many arguments are in $\C{K}$. 

Then, it follows at once from \eqref{n-Lie-diff} that
\[
\d F_jC^{m}(\C{L},V) \subseteq F_jC^{m+1}(\C{L},V).
\]
Accordingly, for the associated graded complex we have
\[
F_jC^{i+j}(\C{L},V)/F_{j+1}C^{i+j}(\C{L},V)=:E_0^{j,i} \cong CL^i(\C{K}_{n-1},C^j(\C{L}/\C{K},V)),
\]
on which the differential map \eqref{n-Lie-diff} induces
\[
\d_0: E_0^{j,i} \to E_0^{j,i+1}
\]
that may be identified with \eqref{Leibniz-diff}, regarding $C^j(\C{L}/\C{K},V)$ as  the representation of $\C{K}_{n-1}$ through \eqref{Leibniz-action-on-complex}. Indeed, setting $f_r\in CL^r(\C{L}_{n-1},C^s(\C{L},V))$ by
\[
f_r(x_1,\ldots,x_{r})(x_{r+1},\ldots,x_{r+s-1}) := f(x_1,\ldots, x_{r+s-1},y)
\]
for any $f \in C^{r+s}(\C{L},V)$, we have 
\begin{align*}
& \d f (x_1,\ldots, x_{r+s},y) = \Big((df_{r+1})(x_1,\ldots,x_{r+1})\Big)(x_{r+2},\ldots,x_{r+s},y) + \\
&\hspace{2cm} \sum_{r+2\leq i < j \leq r+s}\, (-1)^i\, f(x_1,\ldots,\widehat{x_i},\ldots,  [x_i,x_j], \ldots, x_{r+s},y) +\\
 & \hspace{1cm}  \sum_{i=r+2}^{r+s}\, (-1)^i\, f(x_1,\ldots,\widehat{x_i},\ldots,x_{r+s},[x_i^1,\ldots,x_i^{n-1},y]) +\\
  & \hspace{1cm}\sum_{i=r+2}^{r+s}\, (-1)^{i+1} \,\mu(x_i^1,\ldots,x_i^{n-1})\Big( f(x_1,\ldots,\widehat{x_i},\ldots,x_{r+s},y)\Big) +\\
 & \sum_{i=1}^{n-1}\,(-1)^{n-1+{r+s}+i}\, \mu(x_{r+s}^1,\ldots,\widehat{x_{r+s}^i},\ldots,x_{r+s}^{n-1},y) \Big(f(x_1,\ldots,x_{{r+s}-1}, x_{r+s}^i)\Big) = \\
 & \Big((df_{r})(x_1,\ldots,x_{r+1})\Big)(x_{r+2},\ldots,x_{r+s},y) + (-1)^{r+1}\, \Big(\d (f_{r+1}(x_1,\ldots,x_{r+1}))\Big)(x_{r+2},\ldots,x_{r+s},y),
\end{align*}
that is,
\begin{equation}\label{delta-d-comp}
(\d f)_{r+1} (x_1,\ldots, x_{r+1}) = (df_{r})(x_1,\ldots,x_{r+1}) + (-1)^{r+1}\,\d (f_{r+1}(x_1,\ldots,x_{r+1})).
\end{equation}
As such,
\begin{equation}\label{1st-page-spect-seq}
E_1^{j,i} \cong HL^i(\C{K}_{n-1},C^j(\C{L}/\C{K},V)).
\end{equation}

We thus conclude the following.

\begin{corollary}\label{coroll-cohom-gen-der-ext}
Given a $n$-Lie algebra $\C{L}$, and a generalized derivation $D:\C{L}_{n-1}\to\C{L}$, let $\C{L}\oplus k$ be the generalized derivation extension of $\C{L}$ by $k$. Then,
\[
H^m(\C{L}\oplus k, V) \cong H^m(\C{L}, V),
\]
for any representation $V$ of $\C{L}\oplus k$, and any $m\geq0$.
\end{corollary}

\begin{proof}
It follows from \eqref{1st-page-spect-seq}, for $k \subseteq \C{L}\oplus k$, that
\[
E_1^{j,i} \cong \begin{cases}
C^j(\C{L},V) & i = 0, \\
0 & i > 0,
\end{cases} 
\]
from which the result follows.
\end{proof}

\subsection{The second page relative to an ideal}~

Let $\C{L}$ be a $n$-Lie algebra, and let $\C{K}\subseteq \C{L}$ be an ideal. That is, if $\C{L}$ is a $n$-Lie algebra through $[\,,\ldots,\,]:\C{L}_n \to \C{L}$, then for $y_1,\ldots, y_n \in \C{L}$ we have $[y_1,\ldots,y_n]\in\C{K}$ if $y_p \in \C{K}$ for some $1\leq p \leq n$. It then follows that $\C{L}/\C{K}$ is a $n$-Lie algebra through
\[
[\overline{y_1}, \ldots, \overline{y_n}] := \overline{[y_1,\ldots,y_n]}.
\]

We next discuss the $\C{L}/\C{K}$ representation on $HL^i(\C{K}_{n-1},V)$.

\begin{proposition}\label{prop-L/K-action}
Let $\C{L}$ be a $n$-Lie algebra, and let $\C{K}\subseteq \C{L}$ be an ideal that commutes with $\C{L}-\C{K}$, that is, for any $z_1,\ldots, z_n\in \C{L}$,
\[
[z_1,\ldots,z_n]=0
\]
unless $z_1\wdots z_n\in \C{K}_{n}$, or $z_1\wdots z_n\in \wedge^n(\C{L}-\C{K})$. Let also $(V,\mu)$ be a representation of $\C{L}$ so that 
\[
\mu(z)(v) = 0
\]
unless $z:=z_1\wdots z_n\in \C{K}_{n}$, or $z:=z_1\wdots z_n\in \wedge^n(\C{L}-\C{K})$. Then, for any $[f]\in HL^i(\C{K}_{n-1},V)$,
\begin{align*}
& \eta(\overline{z})(f) := \eta(z)(f), \\
& \Big(\eta(z)(f)\Big)(x_1,\ldots,x_i) := \mu(z) (f(x_1,\ldots,x_i)) - \sum_{k=1}^i\,f(x_1,\ldots, [z,x_k],\ldots,x_i)
\end{align*}
determines a representation of $\C{L}/\C{K}$ on $HL^i(\C{K}_{n-1},V)$.
\end{proposition}

\begin{proof}
It suffices, in view of Example \ref{rep-Hom(V,W)}, to verify that $\C{K}_{n-1}$ annihilates $HL^i(\C{K}_{n-1},V)$, which follows from the observation
\begin{equation}\label{L/K-action-on-jf-II}
\eta(z)(f) =(df)_z + d(f_z) = d(f_z),
\end{equation}
see also \cite[Coroll. 4]{daletskii1997leibniz}, for any $z\in \C{K}_{n-1}$.
\end{proof}

\begin{remark}
Let us note that \eqref{L/K-action-on-jf-II} is automatically satisfied for Lie algebras; as in the case of Lie algebras we just need to verify that $\C{K}$ annihilates $HL^i(\C{K}_{n-1},V)$, for which the assumptions of Proposition \ref{prop-L/K-action} are not needed.
\end{remark}

Next, given any $f\in CL^i(\C{K}_{n-1},C^j(\C{L}/\C{K},V))$ we introduce ${}_jf \in C^j(\C{L}/\C{K},CL^i(\C{K}_{n-1},V))$ by
\[
{}_jf(x_{i+1},\ldots, x_{i+j-1},y)(x_1,\ldots,x_i) := f(x_1,\ldots,x_{i+j-1},y).
\]
Then, rearranging the terms in \eqref{delta-d-comp}, and keeping in mind that $\C{K}\subseteq \C{L}$ is an ideal, for any $f\in CL^i(\C{K}_{n-1},C^j(\C{L}/\C{K},V))$ we observe 
\begin{align*}
& (\d f)(x_1,\ldots,x_i,\overline{x_{i+1}},\ldots, \overline{x_{i+j}},\overline{y}) = \\
& \sum_{1\leq r \leq s \leq i}\,(-1)^r\,f(x_1,\ldots, \widehat{x_r},\ldots, [x_r,x_s],\ldots, x_i,\overline{x_{i+1}},\ldots, \overline{x_{i+j}},\overline{y}) + \\
& \sum_{i+1\leq r \leq s \leq i+j}\,(-1)^r\,f(x_1,\ldots,x_i,\overline{x_{i+1}},\ldots, [\overline{x_{r}},\overline{x_{s}}],\ldots, \overline{x_{i+j}},\overline{y}) + \\
& \sum_{r=i+1}^{i+j}\,(-1)^{r+1}\,\mu(x_r)\Big(f(x_1, \ldots, x_i,\overline{x_{i+1}},\ldots, \widehat{\overline{x_r}},\ldots,\overline{x_{i+j}},\overline{y})\Big) + \\
& \sum_{p=1}^{n-1}\,(-1)^{n-1+i+j+p}\,\mu(x_{i+j}^1,\ldots, \widehat{x_{i+j}^p},\ldots,x_{i+j}^{n-1},y)\Big(f(x_1,\ldots,x_i,\overline{x_{i+1}},\ldots,\overline{x_{i+j-1}},\overline{x_{i+j}^p})\Big) = \\
& (-1)^{(i-1)(n-1)(j(n-1)+1)}\,d({}_{j+1}f(\overline{x_{i+1}},\ldots,\overline{x_{i+j}},\overline{y}))(x_1,\ldots,x_i) + \\
& \hspace{3cm} (-1)^{i(n-1)((j-1)(n-1)+1)+i}\,(\d {}_jf)(\overline{x_{i+1}},\ldots,\overline{x_{i+j}},\overline{y})(x_1,\ldots,x_i),
\end{align*}
that is,
\begin{align*}
& (-1)^{i(n-1)(j(n-1)+1)}\,{}_{j+1}(\d f)(\overline{x_{i+1}},\ldots, \overline{x_{i+j}},\overline{y}) = \\
& \hspace{2cm} (-1)^{(i-1)(n-1)(j(n-1)+1)}\,d({}_{j+1}f(\overline{x_{i+1}},\ldots,\overline{x_{i+j}},\overline{y})) +\\
& \hspace{4cm} (-1)^{i(n-1)((j-1)(n-1)+1)+i}\,(\d {}_{j}f)(\overline{x_{i+1}},\ldots,\overline{x_{i+j}},\overline{y}).
\end{align*}
Accordingly,
\[
\d_1:E_1^{j,i}\to E_1^{j+1,i}
\]
induces
\[
E_2^{j,i} \cong H^j(\C{L}/\C{K},HL^i(\C{K}_{n-1},V)).
\]

\section*{Acknowledgment}

The authors acknowledge the support by T\"UB\.ITAK (the Scientific and Technological Research Council of Turkey) through the project "Matched pairs of Lagrangian and Hamiltonian Systems" with the project number 117F426.

\bibliographystyle{plain}
\bibliography{references}{}

\end{document}